\documentclass[A4,12pt]{article}
\usepackage{a4wide} 
\usepackage{url,hyperref,comment}
\usepackage{graphicx,amsfonts,amsmath,amsthm,amssymb}
\usepackage[ruled]{algorithm2e}

\newcommand{\R}{\mathbb{R}}
\newcommand{\E}{\mathbb{E}}
\newcommand{\N}{\mathbb{N}}

\newcommand{\1}{\mathbf{1}}

\newcommand{\normal}{\mathcal{N}}
\newcommand{\C}{\mathcal{C}}
\newcommand{\D}{\mathcal{D}}

\renewcommand{\hat}{\widehat}
\renewcommand{\tilde}{\widetilde}
\renewcommand{\d}{\ensuremath {\,\mathrm{d}}}
\renewcommand{\vec}[1]{{\boldsymbol{#1}}}
\renewcommand{\P}{\mathbb{P}}

\newtheorem{lemma}{Lemma}

\newtheorem{proposition}{Proposition}

\newtheorem{remark}{Remark}
\newtheorem{definition}{Definition}
\newtheorem{corollary}{Corollary}

\title{Adaptive Sensing for Estimation of Structured Sparse Signals}
\author{Ervin T\'{a}nczos and Rui~M.~Castro
\thanks{The authors are with the Department of Mathematics, Eindhoven University of Technology, 5600 MB Eindhoven,
The Netherlands (email \url{e.t.tanczos@tue.nl} and \url{rmcastro@tue.nl}). This work was partially supported by NWO Grant 613.001.114.}}

\begin{document}

\maketitle

\begin{abstract}
In many practical settings one can sequentially and adaptively guide the collection of future data, based on information extracted from data collected previously. These sequential data collection procedures are known by different names, such as \emph{sequential experimental design}, \emph{active learning} or \emph{adaptive sensing/sampling}. The intricate relation between data analysis and acquisition in adaptive sensing paradigms can be extremely powerful, and often allows for reliable signal estimation and detection in situations where non-adaptive sensing would fail dramatically.
In this work we investigate the problem of estimating the support of a structured sparse signal from coordinate-wise observations under the adaptive sensing paradigm. We present a general procedure for support set estimation that is optimal in a variety of cases and shows that through the use of adaptive sensing one can: (i) mitigate the effect of observation noise when compared to non-adaptive sensing and, (ii) capitalize on structural information to a much larger extent than possible with non-adaptive sensing. In addition to a general procedure to perform adaptive sensing in structured settings we present both performance upper bounds, and corresponding lower bounds for both sensing paradigms.
\end{abstract}


\section{Introduction} \label{sec:intro}

Consider the problem of estimating the support set of an unknown vector (which we refer to as a signal) through noisy coordinate wise measurements. Such a problem may arise in various contexts, ranging from gene expression studies (where one tries to identify genes that are differentially expressed under some specific condition) or network monitoring (where one wishes to either detect or locate anomalous elements in a network). Under non-adaptive sensing paradigms, the most natural way to collect data is to measure each coordinate of the vector with the same accuracy (that is, provided each coordinate of the vector is equally likely to be in the support set). However, what if we have the additional flexibility of also choosing the precision and location of each measurement based on the data collected so far? It is not immediately clear how much can be gained by these adaptive sensing strategies over the non-adaptive ones.

In this paper we consider sparse signals, meaning the size of the support set is significantly less than the total number of components of the signal. Sparsity has enjoyed quite a lot of success as a modeling tool \cite{efron:07,donoho:98,genovese:09} . However, in addition to this sparsity assumption one might consider further structural restrictions on the unknown support set. For instance, in network anomaly monitoring one may assume anomalous behavior ``radiates'' from certain networks elements, giving rise to star-shaped patterns on the network graph. In gene expression studies one may expect certain groups of genes that are involved in certain biological mechanisms to share common expression levels among different individuals with the same type of disease. In this case the support set of the gene-expression matrix is a submatrix. Can we use such additional structural knowledge to further increase the performance of support set estimation algorithms? If so, to what extent does such structural knowledge help us? In this work we aim to address the questions above in the framework considered in \cite{DS_Haupt_2011} and \cite{AS_Rui_2012}. In particular we shed light on how adaptive sensing can capitalize on structural information, by providing general and practical algorithms endowed with performance guarantees. Furthermore we show that these algorithms are essentially optimal as we give matching performance lower bounds.

Inference methods where data collection is sequential and adaptive have been studied in many different disciplines. In statistics, these are often referred to as sequential design of experiments \cite{SLRT_Wald_1945,Seq_Chernoff_1959,Fedorov_1972,El-Gamal_1991,Hall_2003,Blanchard_2005,Seq_Testing_Limits_Malloy_2011,Seq_Testing_Malloy_2011}. In machine learning and computer science, these methodologies are often known as active learning, and are of great importance in applied settings \cite{Power_Novak_1996,Cohn_1996,Dasgupta_2005,Sample_Comp_Bounds_Dasgupta_2005,Perceptron_Based_Dasgupta_2005,Willett_2005,Minimax_Castro_2008,Balcan_2009,Koltchinskii_2010,Hanneke_2011}. The effect that structural information regarding the unknown parameter of interest has on statistical inference is an important question, and was studied in e.g. \cite{Trail_Arias_2008,Comb_Testing_Lugosi_2010,Butucea_2011} in non-adaptive sensing settings, and investigated for a particular class of structural information in \cite{Balakrishnan_2012} in an adaptive sensing setting.

The paper is structured as follows. Section~\ref{sec:setting} describes the framework we are considering in this work. In Section~\ref{sec:procedure} we introduce a general procedure for support set estimation. We analyze the performance of the procedure in Section~\ref{sec:performance}. The performance limitations of any support estimation recovery procedure is investigated in Section~\ref{sec:lower_bounds}. Finally we provide concluding remarks in Section~\ref{sec:conclusion}.


\section{Problem Setting} \label{sec:setting}

Let $\vec{x}=(x_1,\ldots,x_n)\in\R^n$ be a vector of the form
\begin{equation}\label{eqn:signal_model}
x_i= \left\{ \begin{array}{ll}
\mu & \textrm{if $i \in S$}\\
0 & \textrm{if $i \notin S$}
\end{array}\right. \ ,
\end{equation}
where $\mu>0$, and $S$ is a subset of $\{1,\ldots,n\}$. We refer to the vector $\vec{x}$ as the \emph{signal}, and to $S$ as the \emph{signal support} or the \emph{significant components} of the signal. The latter is our main object of interest, as neither $\vec{x}$ and $S$ are directly available. We are allowed to collect multiple noisy measurements of each individual component of $\vec{x}$, namely
$$Y_j = x_{A_j} + \Gamma_j^{-1/2} W_{j}\ , \  j=1,2,\ldots\ ,$$
where $j$ indexes the $j$th measurement. For each measurement we can choose $A_j$, the entry of $\vec{x}$ to be measured, and the corresponding \emph{precision} of the measurement $\Gamma_j> 0$. Finally $W_{j} \sim \normal(0,1)$ are independent and identically distributed standard normal random variables. Also for any given $j$, $W_j$ is independent of $\{A_i,\Gamma_i\}_{i=1}^{j}$. Under the adaptive sensing paradigm $A_j$ and $\Gamma_j$ are allowed to be functions of the past observations $\{Y_i,A_i,\Gamma_i\}_{i=1}^{j-1}$. This model is only interesting if one includes some constraint on the total amount of precision available. Let $\P_S$ denote the joint probability distribution of $\{Y_i,A_i,\Gamma_i\}_{i=1}^{\infty}$ and $\E_S$ denote the expectation with respect to $\P_S$. In this paper we require that
\begin{equation}\label{eqn:budget}
\E_S \left(\sum_{j} \Gamma_j \right) \leq m\ ,
\end{equation}
where $m$ is our total precision budget, given in advance. This constraint arises naturally in many practical settings, and can be viewed as a total measurement time constraint in sensing modalities where precision is directly proportional to the amount of time necessary to collect a measurement (see for instance \cite{sampta:13}). Finally, using the collected data we construct an estimator $\hat S\equiv \hat S(\{Y_i,A_i,\Gamma_i\}_{i=1}^{\infty})$ that is desirably as close to $S$ as possible.

At a first glance it might seem that the model \eqref{eqn:signal_model} is overly restrictive, as all the significant components of $\vec{x}$ have exactly the same value $\mu$. However, the results in this paper can be generalized to sparse signals with non-zero significant components of arbitrary signs and magnitudes, provided the minimum magnitude of these is large enough. For the sake of clarity and simplicity we do not consider this extension here, but refer the reader to \cite{AS_Rui_2012} for details on how this can be done.

In this work the primary focus is on adaptive sensing algorithms. However, for comparison purposes we will also consider non-adaptive sensing inference, which means $\{A_i,\Gamma_i\}_{i=1}^{\infty}$ must be chosen before any observations are collected. In other words, non-adaptive sensing requires $\{A_i,\Gamma_i\}_{i=1}^{\infty}$ to be statistically independent from $\{Y_i\}_{i=1}^{\infty}$.


\subsection{Inference Goals} \label{sec:goals}

Since our goal is to characterize the fundamental limitations of adaptive sensing, we will assume $\mu$ is known, in addition to $n$ and $m$. Therefore the only unknown quantity is the signal support $S$. Our aim is to construct adaptive sensing methodologies that are able to estimate $S$. This is only possible if the signal magnitude $\mu$ is large enough. Furthermore, it is reasonable and desirable to make some concrete assumptions about $S$, namely that the signal has a sparse support meaning the cardinality of $S$ is small, and also extra structural assumptions. All these can be formalized by assuming $S$ belongs to some class $\C$ of subsets of $\{1,\ldots,n\}$.

There are various ways one can define \emph{reliable} estimation of a support set \cite{sampta:13}. In this work we consider the worst-case Hamming-distance as an error metric of our estimator. Let $\hat S\equiv \hat S(\{Y_i,A_i,\Gamma_i\}_{i=1}^{\infty})$ be a specific estimator. We wish to ensure that for a given $\varepsilon>0$, 
\begin{equation} \label{eqn:goal}
\max_{S \in \C} \E_S (|\hat S  \triangle S|) \leq \varepsilon \ ,
\end{equation}
where $\hat S \triangle S$ is the symmetric set difference of $\hat S$ and $S$, and $|\cdot|$ denotes the cardinality of a set. In words, we require the expected number of errors to be less than $\varepsilon$, regardless of the true unknown support set $S$.

One can also consider a slightly less stringent metric, namely the probability of falsely identifying the support set, that is $\P_S (\hat{S} \neq S)$. Note that we have
\begin{equation}\label{eqn:hamming_vs_prob}
\P_S (\hat{S} \neq S)\leq \E_S (|\hat S \triangle S|) \leq 2 |S| \ \P_S (\hat S \neq S) \ ,
\end{equation}
where second inequality holds provided $|\hat S|=|S|$ (this property holds for all the estimators we consider). According to this we are able to control the expected number of errors of a procedure by controlling the probability of error. This is exactly what we do, so the analysis of the procedure we propose will be applicable to both error metrics. In addition, we also derive lower bounds in terms of expected Hamming-distance. Whenever possible we also provide lower bound in terms of probability of error.

Naturally, the choice of support set class $\C$ plays a crucial role. The classes we consider in this paper fall into two categories: (i) all support sets of cardinality $s$, which we call \emph{$s$-sets}. This is the maximal class for a given level of sparsity $s$, thus we refer to this class, or the union of such classes with different values of $s$, as the \emph{unstructured case}. In contrast, other classes we consider are more stringent as the sets are structurally restricted. For instance, the class of $s$-intervals, which are sets of the form $\{i,i+1,\ldots,i+s-1\}$ for $i\in\{1,\ldots,n-s+1\}$. This class of sets is much smaller than the class of $s$-sets, and therefore we expect the support recovery task to be significantly easier. We use the umbrella term \emph{structured case} for such classes. In particular, we study the following classes:
\begin{itemize}
\item{\textbf{$s$-sets:} any subset of $\{1,\ldots,n\}$ with size $s$}
\item{\textbf{$s$-intervals:} sets consisting of $s$ consecutive elements of $\{1,\ldots,n\}$}
\item{\textbf{unions of $s$-intervals:} unions of $k$ disjoint $s$-intervals}
\item{\textbf{$s$-stars:} any star of size $s$ in a complete graph (where the edges of the graph are identified with $\{1,\ldots,n\}$)}
\item{\textbf{unions of $s$-stars:} unions of $k$ disjoint $s$-stars}
\item{\textbf{$s$-submatrices:} any submatrix of size $s$ of a $\sqrt{n}\times\sqrt{n}$ matrix}
\end{itemize}

\begin{table}[h]
\smaller
\begin{center}
\begin{tabular}{l|c|c|c}
& Non-Adaptive Sensing & \multicolumn{2}{c}{Adaptive Sensing}\\
& (necessary) & (necessary) & (sufficient)\\
\hline & & &\\
$s$-sets & $\mu \sim \sqrt{\log n}$ & $\mu \sim \sqrt{\frac{n}{m}\log s}$ & $\mu \sim \sqrt{\frac{n}{m}\log s}$\\
unions of $k$ disjoint $s$-intervals & $\mu \sim \sqrt{\frac{n}{sm}\log \frac{n}{ks}}$ & $\mu \sim \sqrt{\frac{n}{sm}\log ks}$ & $\mu \sim \sqrt{\frac{n}{sm}\log ks}$\\
unions of $k$ disjoint $s$-stars & $\mu \sim \sqrt{\frac{n}{m}\log \frac{\sqrt{n}}{ks}}$ & $\mu \sim \sqrt{\frac{n}{sm}\log ks}$ & $\mu \sim \sqrt{\frac{n}{sm}\log^2 ks}$\\
$s$-submatrices of an $\sqrt{n}\times\sqrt{n}$ matrix & $\mu \sim \sqrt{\frac{n}{\sqrt{s}m}\log \frac{n}{s}}$ & $\mu \sim \sqrt{\frac{n}{sm}\log s}$ & $\mu \sim \sqrt{\frac{n}{sm}\log^2 s}$\\
\end{tabular}
\end{center}
\caption{Summary of scaling laws for the signal magnitude $\mu$ (constants omitted) which are necessary/sufficient for $\max_{S\in\C} \E(\hat S\triangle S)\to 0$ as $n\to\infty$, where $\C$ denotes the corresponding class of support sets. All the results assume sparsity, meaning both $s=o(\sqrt{n})$ and $ks=o(\sqrt{n})$ and $n\to\infty$.}
\label{table}
\end{table}

Table~\ref{table} summarizes the results obtained in this paper, stated in terms of asymptotic behavior when the signal dimension $n$ is large and the support set (of size $s$ or $ks$) is small. Note that most results in the paper are not asymptotic in nature, and furthermore the constant factors in the scaling laws are also accounted for. Nevertheless the results become easier to state and interpret in asymptotic terms.

A first point to notice is that the necessary condition for non-adaptive sensing always includes a $\sqrt{\log n}$ factor, regardless of class considered. This factor is essentially due to the extreme value properties of Gaussian random variables. Note, however, that for adaptive sensing that factor is replaced by a $\sqrt{\log ks}$ term (where $k=1$ for the class of $s$-sets and $s$-submatrices). This means that adaptive sensing can always mitigate the effect of measurement noise. This is particularly interesting when $m=n$ (or more generally $m$ is proportional to $n$) meaning that one can make, on average, one measurement of precision one per signal entry. In that case the dependence on the extrinsic dimension $n$ vanishes completely when considering adaptive sensing, as opposed to non-adaptive sensing where the factor $\sqrt{\log n}$ is ever-present. However, the gains of adaptive sensing when structure is present can sometimes be much more remarkable. For discussion purposes consider the case $m=n$: for the class of unions of disjoint $s$-stars one gets that $\mu\sim \sqrt{\log n}$ is necessary for non-adaptive sensing, but if suffices that $\mu\sim\sqrt{(1/s)\log^2 ks}$ for adaptive sensing. Therefore, apart from logarithmic factors, there is also a factor $\sqrt{1/s}$ reduction on signal magnitude with adaptive sensing. This can be rather beneficial, for instance when $s\sim n^{\beta}$ for some $0<\beta<1/2$. These gains stem from the strong structural constraints in the class, which can be exploited by adaptive sensing strategies. However, as the cardinality of this class is still very large it renders the structural information almost useless for non-adaptive sensing. A similar situation happens for the $s$-submatrices class, although the gains there are less dramatic (apart from logarithmic factors there is a factor $s^{-1/4}$ reduction in signal magnitude). Finally, for the class of unions of $s$-intervals such structural gains are not present (although the logarithmic factors are still significantly improved). In summary, adaptive sensing can both remove the dependence on the extrinsic dimension $n$ due to noise (which is reflected in the logarithmic terms), and further improve the signal magnitude scaling laws (compared to non-adaptive sensing) when further structural information is present.

\begin{remark}
In this paper we consider only Gaussian observation noise. However, all the results in this paper can be generalized to non-Gaussian noise models, and will lead to different scaling laws. Nevertheless, the qualitative comparison between adaptive and non-adaptive sensing will remain essentially the same.
\end{remark}



\section{A General Adaptive Sensing Estimation Procedure} \label{sec:procedure}

At the core of the problem setting we described is the issue of noise and measurement uncertainly, which is embodied by the precision budget in \eqref{eqn:budget}. Without this restriction the inference task is much easier, as one can make noiseless (infinite precision) measurements. Nevertheless, a sharp distinction between adaptive and non-adaptive sensing is still present, meaning one can still devise powerful adaptive sensing procedures. This gives rise to a simple, yet very powerful idea: take the \emph{noiseless} adaptive sensing procedures and transform them to be robust to noise. Our general approach hinges precisely on this ``robustification'' of noiseless procedures, which we refer to as \emph{noiseless-case algorithms} or \emph{algorithms for the noiseless case}. When the noiseless case algorithm observes an entry of $\vec{x}$, we take multiple noisy measurements of that entry and perform a sequential hypothesis test to decide whether the entry in question is zero or not. Then we use the result of the test as a surrogate for the noiseless observation. If we ensure these tests have small enough probabilities of error, we can recover the support set with high probability. By carefully controlling these error probabilities we can also control the expected Hamming-distance performance of the devised estimator. To better illustrate the ideas we will make use of two running scenarios (corresponding to two different classes $\C$): (i) the class of \emph{$s$-sets}, that is when the support set $S$ is an arbitrary subset of $\{1,\ldots,n\}$ with cardinality $s$; the class of \emph{$s$-intervals}, that is all sets consisting of $s$ consecutive elements of $\{1,\ldots,n\}$.

\subsection{Noiseless-case algorithms}

An algorithm based on coordinate wise sampling for the noiseless case can be described as follows. At each time $t \in \N$, the algorithm either collects an observation of a coordinate of the signal vector $\vec{x}$, or stops and returns the estimate $\hat S$ for the support set $S \in \C$. The observation collected at time $t$ is denoted by $\tilde{Y}_t = \1 \{ x_{Q_t} \neq 0 \}$, where $Q_t \in \{ 1, \dots ,n \}$ determines the coordinate of $\vec{x}$ that we sample at time $t$, and $\1\{\cdot\}$ denotes the usual indicator function. We call $Q_t$ the \emph{query} at time $t$, and it plays a role analogous to that of $A_t$ in the problem description. In case the component indexed by $Q_t$ is a signal component the value of $\widetilde{Y}_t$ is 1, otherwise it is 0. Note that if $\1_S (.): \{ 1,\dots ,n \} \to \{ 0,1 \}$ denotes the indicator function of the support set $S$, the observations can be written as $\tilde{Y}_t = \1_S (Q_t)$. After taking a number of observations we may decide to stop and return the estimator of the support set. $T$ denotes the stopping time for the procedure and $\hat{S}$ denotes the estimate of the support set $S \in \C$.

To fully describe such an algorithm, one needs to give the queries $Q_t: t = 1,2, \ldots$, a stopping time $T$ and a rule for constructing $\hat S$. The query $Q_t: t = 1,2,\ldots$ is a measurable function of $\{Q_i,\tilde{Y}_i \}_{i=1}^{t-1}$ mapping to $\{1,\ldots,n\}$. It determines the coordinate of $\vec{x}$ that we wish to sample at time $t$. The query $Q_t$ can be random, so that randomized procedures can be considered. Note that because the observations are noiseless, it is unnecessary to sample any coordinate of $\vec{x}$ more than once, and therefore we only consider procedures satisfying this property.

The stopping time $T$ is the possibly random time at which we stop sampling and return an estimate of the support set. Thus $T$ is an $\N$ valued measurable function of the filtration generated by $\{Q_i,\tilde{Y}_i \}_{i=1,2,\ldots}$ which we define in the following way
\begin{equation}\label{def:T}
T= \inf \left\{ t:  \text{ exists at most one }  S' \in \C: \  \tilde{Y}_i = \1_{S'} (Q_i ) \ \forall i \in \{1,\ldots,t \} \right\} \ .
\end{equation}
This means that we consider procedures that stop sampling when there is a unique set in $\C$ that agrees with all the observations, or if there is no such set. Note that $T$ is well defined, and since it is unnecessary to sample any coordinate of $\vec{x}$ more than once in the noiseless case $T \leq n$. Furthermore in the noiseless case the procedure stops when there is exactly one set in $\C$ in line with our observations (since we assume $S \in \C$). However, we will later modify the procedure to be able to handle noise and thus there will be a chance of making errors. Because of this, it is possible that there is no set in $\C$ which is in agreement with all our observations. For this reason we enforce the procedure to stop when this happens to ensure it remains well-defined after the modification.

The estimator $\hat S$ is then defined in the noiseless case as the unique set in $S' \in \C$ which agrees with our observations. Since $T$ is well-defined, $\hat S$ is also well-defined in the noiseless case (and actually $\hat S=S$). However, to ensure that $\hat S$ remains well-defined after we modify this procedure to handle noise we define the estimator as
\begin{equation} \label{def:S_hat}
\hat S = \left\{ \begin{array}{ll}
S' & \text{if}\quad \exists ! \ S' \in \C: \ \tilde{Y}_t = \1_{S'} (Q_t ) \ \forall t \in \{ 1, \dots ,T \} \\
\emptyset & \text{otherwise} \\
\end{array} \right. \ .
\end{equation}

To illustrate how such a procedure may look like, consider the examples of \emph{$s$-sets} and \emph{$s$-intervals}. In the first case consider a deterministic procedure, which samples every coordinate one after the other. That is let $Q_t = t, \ \forall t=1,\ldots,T$. The procedure will stop once there is a unique $s$-set that agrees with our observations. This means it will stop when either it finds all the elements of $S$ or when $T=n-1$ (because the last element is uniquely determined by all the previous ones). Once sampling has stopped, the estimate of the support set is $\hat S=S'$, where $S' \in \C$ is the unique set satisfying $\tilde{Y}_t = \1_{S'} (Q_t ), \ \forall t \in \{ 1, \ldots ,T \}$. Next, consider the class of \emph{$s$-intervals}. Consider a randomized procedure consisting of two phases. In the first phase sample random coordinates of the vector $\vec{x}$ until a non-zero coordinate is found. In the second phase search for the left endpoint of the interval by sampling coordinates one after the other, which are to the left of the one containing signal that was found in the previous phase. The interval $S$ is exactly determined either when a 0 is found in the second phase, or when all the $s$ signal components are found. Formally, denoting by $\mathrm{Unif}$ the discrete uniform distribution, the procedure can be written as $Q_t \sim \mathrm{Unif}\left( \{1,\dots ,n \} \setminus \{ Q_1,\dots ,Q_{t-1} \} \right), \ \forall t \leq T'$, where
$$T' = \inf \left\{ t: \tilde{Y}_t =1 \text{ and } \tilde{Y}_j = 0 \quad \forall j=1,\dots ,t-1 \right\}\ ,$$
and $Q_t = Q_{t-1} -1, \ \forall t= T'+1,\dots ,T$. The estimator $\hat S$ is defined as before as the unique set compatible with the observations. Note that no claim is made about whether this procedure is optimal in any sense. In particular it is possible to construct a procedure which takes less number of steps in expectation than this one, for instance by performing a binary search in the second phase.

\subsection{From the noiseless to the noisy case}

Assume now that one has a noiseless-case procedure. The next step is to translate this procedure to the noisy case, to handle the situation when the observations are contaminated by noise \eqref{eqn:signal_model}, and there is a total precision budget \eqref{eqn:budget}. The main idea is to replace each query $Q_t$ by multiple observations of the entry of $\vec{x}$ indexed by $Q_t$, and perform a hypothesis test to assess whether the component corresponding to that query is zero or not. Specifically, we will set type I and type II error probabilities $\alpha_t$ and $\beta_t$ for each $Q_t$, perform a Sequential Likelihood Ratio Test (SLRT) with these error probabilities, and use its result as a surrogate for $\tilde{Y}_t$. How to properly choose the error probabilities $\alpha_t$ and $\beta_t$ depends on the specific problem at hand, but for now assume these are simply given to us.

The procedures we propose have the nice property that all observations are made with the same precision $\Gamma$, namely $\Gamma_j = \Gamma >0 \  \forall j \in \N$. This is not at all restrictive, provided $\Gamma$ is relatively small, as justified by Proposition~\ref{prop:SLRT}. For the first query $Q_1$ set the target type I and type II error probabilities to be $\alpha_1$ and $\beta_1$ respectively. The SLRT collects observations
$$Y_j = x_{Q_1} + \Gamma^{-1/2} W_{j}\ , \  j=1,\dots ,N_1 \ ,$$
where $N_1$ is an appropriate stopping time defined as follows. Let $f_0 (\cdot)$ and $f_1 (\cdot)$ denote the density of the observations when $Q_1\notin S$ and $Q_1\in S$ respectively. Define the likelihood ratio
\begin{equation}\label{eqn:loglikeratio}
\bar z_k = \sum_{j=1}^{k} \log \frac{f_1 (Y_j)}{f_0 (Y_j)}\ .
\end{equation}
The stopping time $N_1$ is defined as
$$N_1 = \inf \left\{ k \in \N: \bar z_k \notin (l_1, u_1)  \right\} \ ,$$
where $l_1 < 0 < u_1$ are chosen so that $\P\left( Z_{N_1} \geq u_1 | Q_1\notin S \right) \approx \alpha_1$ and $\P\left( Z_{N_1} \leq l_1 | Q_1\in S \right) \approx \beta_1$. Once $N_1$ observations have been collected a decision is made regarding whether or not $Q_t$ belongs to the support set. Namely we define the test function $\Psi_1$ as
$$\Psi_1 = \left\{ \begin{array}{ll}
0 & \text{ if } \bar z _{N_1} \leq l_1 \\
1 & \text{ if } \bar{z}_{N_1} \geq u_1
\end{array} \right. \ .$$
We use the value of $\Psi_1$ as a surrogate for $\tilde{Y}_1$ in the noiseless case procedure. This then determines the next query $Q_2$. Again we perform another SLRT by taking observations of the coordinate $\vec{x}_{Q_2}$. We set the type I and type II error probabilities to be $\alpha_2$ and $\beta_2$, determine upper and lower stopping boundaries $l_2,u_2$, perform the test resulting in $\Psi_2$ which we use as a surrogate for $\tilde{Y}_2$, and so on. We continue in this manner until the condition for the stopping time $T$ of the noiseless case procedure is met, and return the corresponding estimate $\hat S$. The whole procedure is summarized in Algorithm~\ref{algo}.

\begin{algorithm}[h]
\caption{General Adaptive Sensing Support Estimation}\label{algo}
\SetKw{KwParameters}{Input:}
\SetKw{KwInitialization}{Initialization:}
\SetKw{KwReturn}{Terminate:}
\KwParameters{}\\
$\quad$\textbullet\ A noiseless procedure characterized by: queries $Q_j,\  j=1,2,\dots$, stopping criterion $T$, and estimator $\hat{S}$\\
$\quad$\textbullet\ Precision parameter $\Gamma>0$\\
$\quad$\textbullet\ Type I and II error probabilities $\alpha_t$ and $\beta_t$ corresponding to query $Q_t$\\
\For{$t\leftarrow 1$ \KwTo $\cdots$}
	{
	 Perform an SLRT for entry $x_{Q_t}$ with error probabilities $\alpha_t,\beta_t$ resulting in $\Psi_t$\\
	 Set $\tilde{Y}_t = \Psi_t$\\
	 If $T = t$ stop and return $\hat S$\\
	}
\end{algorithm}


It is important to notice that the procedure is well defined. In particular, each of the SLRTs terminates with probability one, as shown in Proposition~\ref{prop:SLRT}. Furthermore, by the definition of $T$ (see \eqref{def:T}) the entire procedure is guaranteed to terminate with probability one, even if some of the SLRTs result in errors (meaning $\Psi_t \neq \textbf{1}_S (Q_t)$). Finally, the definition \eqref{def:S_hat} ensures $\hat S$ is also well defined in the event of errors.



\section{Performance Upper Bounds} \label{sec:performance}

In this section we use the procedure outlined in the previous section to characterize attainable inference limits in various settings. The SLRT is at the heart of our procedure, and therefore we begin by deriving some important properties these satisfy. We then move on to the analysis of the full procedure.


\subsection{Analysis of the SLRTs} \label{sec:SLRT}

Most tools used in our analysis stem from the seminal work in \cite{SLRT_Wald_1945}. However, some of these results have to be specialized for our setting. Consider a SLRT that we use to decide between the two simple hypotheses $H_0$ and $H_1$. We collect independent and identically distributed measurements $y_1, y_2,\ldots$, where $y_1 \sim \normal(0 , \Gamma^{-1})$ under $H_0$ and $y_1 \sim \normal(\mu , \Gamma^{-1})$ under $H_1$. We set as target type I and type II error probabilities $\alpha$ and $\beta$ respectively. These determine upper and lower stopping boundaries which we denote by $l=\log \frac{\beta}{1- \alpha}$ and $u=\log \frac{1-\beta}{\alpha}$. Recall the definition of the log-likelihood ratio in \eqref{eqn:loglikeratio}, and define the stopping time $N_{\Gamma}$ as
$$N_{\Gamma} = \inf \left\{ k \in \N: \ \bar z_k \notin (l, u)  \right\} \ ,$$
where $f_0$ and $f_1$ are the densities of $y_1$ under $H_0$ and $H_1$ respectively, and the subscript $\Gamma$ is meant to emphasize the dependence in $\Gamma$. Finally define the test $\Psi$ as
$$\Psi = \left\{ \begin{array}{ll}
0 & \text{ if } \bar z_{N_{\Gamma}} \leq l \\
1 & \text{ if } \bar z_{N_{\Gamma}} \geq u
\end{array} \right. \ .$$

We know from the theory of SLRTs that $\P\left(N_{\Gamma} < \infty \right) = 1$ (see \cite{SLRT_Wald_1945}), so it is guaranteed the data collection terminates almost surely. We also know that
$$\E_0 (N_{\Gamma}) \geq \frac{1}{- D( \P_0 \| \P_1 )} \left[ \alpha \log \frac{1- \beta}{\alpha} + (1- \alpha ) \log \frac{\beta}{1- \alpha} \right] \ ,$$
and
$$\E_1 (N_{\Gamma}) \geq \frac{1}{D( \P_1 \| \P_0 )} \left[ (1- \beta ) \log \frac{1- \beta}{\alpha} + \beta \log \frac{\beta}{1- \alpha} \right] \ ,$$
where $\P_0$ and $\P_1$ are the distributions of $y_1$ under $H_0$ and $H_1$ respectively, $\E_0$ and $\E_1$ are the expectations with respect to $\P_0$ and $\P_1$ respectively and $D(\cdot\|\cdot)$ is the Kullback-Leibler divergence of two distributions. Since $\P_0$ and $\P_1$ are normal distributions we have $D( \P_1 \| \P_0 ) = D( \P_0 \| \P_1 )=\Gamma \mu^{2} /2$ and therefore
\begin{equation}\label{eqn:prec_lower_0}
\Gamma \E_0 (N_{\Gamma}) \geq \frac{2}{\mu^{2}} \left[ \alpha \log \frac{\alpha}{1- \beta} + (1- \alpha ) \log \frac{1- \alpha}{\beta} \right]
\end{equation}
and
\begin{equation}\label{eqn:prec_lower_1}
\Gamma \E_1 (N_{\Gamma}) \geq \frac{2}{\mu^{2}} \left[ (1- \beta ) \log \frac{1- \beta}{\alpha} + \beta \log \frac{\beta}{1- \alpha} \right] \ .
\end{equation}
The derivation of these lower bounds goes roughly as follows. The cumulative log-likelihood $\bar{z}_k, \ k=1,2,\ldots$ is a discrete time stochastic process. The process terminates when it leaves the interval $(l, u)$. By assuming that the process exactly hits the boundaries of this interval we can get the lower bounds above. In reality the log-likelihood ratio will never be exactly equal to $l$ and $u$. However, when the precision $\Gamma$ is small, the increments to the stochastic process $\bar z_k$ are also small, and so this process will nearly hit the exact boundaries of the interval $(l,u)$. This in turn means the above lower bounds should be attainable when $\Gamma$ approaches zero. This is indeed the case, as stated in the following result, which is proved in the Appendix.

\begin{proposition} \label{prop:SLRT}
Let $\alpha_{\Gamma} = \P_0 (\Psi =1)$ and $\beta_{\Gamma} = \P_1 (\Psi =0)$ be, respectively, the type I and II error probabilities of the SLRT. Then
$$\alpha_{\Gamma} \to \alpha \quad\text{ and }\quad  \beta_{\Gamma} \to \beta$$
as $\Gamma \to 0$. Furthermore
$$\Gamma \E_0 (N_{\Gamma}) \to \frac{2}{\mu^{2}} \left[ \alpha \log \frac{\alpha}{1- \beta} + (1- \alpha ) \log \frac{1- \alpha}{\beta} \right]$$
and
$$\Gamma \E_1 (N_{\Gamma}) \to \frac{2}{\mu^{2}} \left[ (1- \beta ) \log \frac{1- \beta}{\alpha} + \beta \log \frac{\beta}{1- \alpha} \right]\ ,
$$
as $\Gamma \to 0$.
\end{proposition}

This result shows that, provided the precision of each individual measurement is relatively small we have a precise characterization of the total expected precision used by the SLRT. Furthermore it shows that the lower bounds on the expected amount of precision used by the SLRT with error probabilities $\alpha , \beta$ can be achieved in the limit when $\Gamma \to 0$. Thus, when analyzing the performance of our procedures in terms of expected precision used we can use these lower bounds to calculate the expected precision used by the SLRTs. Note that we are interested in the case when $\alpha$ and $\beta$ are small. Thus, to make the discussion even more transparent we note that when $\alpha$ and $\beta$ are both at most 1/2 we have
\begin{equation} \label{eqn:prec_upper_0}
\frac{2}{\mu^{2}} \Big[ \alpha \log \frac{\alpha}{1- \beta} + (1- \alpha ) \log \frac{1- \alpha}{\beta} \Big] \leq \frac{2}{\mu^{2}} \log \frac{1}{\beta}
\end{equation}
and
\begin{equation} \label{eqn:prec_upper_1}
\frac{2}{\mu^{2}} \Big[ (1- \beta ) \log \frac{1- \beta}{\alpha} + \beta \log \frac{\beta}{1- \alpha} \Big] \leq \frac{2}{\mu^{2}} \log \frac{1}{\alpha} \ .
\end{equation}
When $\alpha$ and $\beta$ are approximately zero, the inequalities above are essentially tight. In what follows we will use the quantities on the right hand sides when calculating the expected precision used by a SLRT. By the previous proposition if we choose $\Gamma$ to be small enough, these quantities upper bound the expected precision used by the SLRT.


\subsection{General Analysis of Algorithm~\ref{algo}} \label{sec:analysis_procedure}

Now we turn our attention to the analisys of the general procedure of Section~\ref{sec:procedure}. Recall that a procedure for the noiseless case is characterized by queries $Q_t, \ t=1,2,\dots$, a stopping time $T$ which indicates the time when we stop sampling, and the estimator $\hat{S}$. In what follows we consider always the definition of the two last quantities given by \eqref{def:T} and \eqref{def:S_hat}. The queries $Q_t$ will be defined separately for each special case.

Given a certain noiseless case procedure we translate it to the noisy case by replacing each noiseless query $Q_t$ by a surrogate SLRT $\Psi_t$. This requires the specification of type I and type II error probabilities $\alpha_t$ and $\beta_t$ for each of the tests $\Psi_t, \ t=1,2,\ldots$. Naturally, $\alpha_t$ and $\beta_t$ can be, in general, functions of $\{Q_i ,\Psi_i\}_{i=1}^{t-1}$, and we wish to choose them to ensure that the final estimator $\hat S$ satisfies $\E_S (|\hat{S} \triangle S|) \leq \varepsilon, \ \forall S \in \C$ on one hand, and that the total precision budget \eqref{eqn:budget} is not exceeded. Clearly to meet the former goal $\alpha_t$ and $\beta_t$ need to be small enough, while if these are too small the latter goal might not be attained. Therefore we need to make a compromise in setting the error probabilities $\alpha_t, \beta_t$. How to optimally choose $\alpha_t, \beta_t, \ t=1,2,\dots$ depends on the specific procedure under consideration (and class of possible support sets), and it is difficult to get a general answer. However, we will see that in many interesting cases simple and intuitive choices for $\alpha_t, \beta_t$ yield near optimal results.

We illustrate the analysis of the procedure by first considering the unstructured case of all $s$-sets. In the unstructured case the near optimal procedure is very simple, and our choice of $\alpha_t ,\beta_t$ does not depend on $t$, which greatly facilitates the analysis. Formally the class of $s$-sets is defined as
$$\C = \left\{ S \subseteq \{ 1,\dots ,n \} : \ |S|=s \right\} \ .$$
A simple procedure for the noiseless case is defined by taking $Q_t = t, \ t=1,2,\ldots,T$. Because of the sparsity of the signal we expect the majority of the coordinates we sample to be zero, and we know that there are exactly $s$ that are non-zero. So it is sensible to take $\alpha_t \approx \varepsilon /n$ and $\beta_t \approx \varepsilon /s$. We will take the following concrete choice
$$\alpha_t = \varepsilon /2n \quad\text{ and }\quad \beta_t = \varepsilon /2s\ , t=1,2,\ldots\ .$$

In the worst case, for any $S \in \C$ we query all the entries of $\vec{x}$. Using this crude upper bound we get
\begin{align*}
\E_S (|\hat{S} \triangle S|) &\leq \sum_{t=1}^{T} \P_S (\Psi_t \neq \textbf{1}_S(Q_t)) \\
&\leq \sum_{t \notin S} \alpha_t + \sum_{t \in S} \beta_t \\
&\leq n \frac{\varepsilon}{2n} + s \frac{\varepsilon}{2s} \leq \varepsilon  \ .
\end{align*}
Since the inequality above holds for all $S \in \C$ we conclude that the expected number of errors for any $S \in \C$ is at most $\varepsilon$. Furthermore the total amount of precision used by this procedure is
\begin{align*}
\E_S \left( \sum_{j} \Gamma_j \right) &\leq \displaystyle{\sum_{t \notin S}} \frac{2}{\mu^{2}} \log \frac{2s}{\varepsilon} + \displaystyle{\sum_{t \in S}} \frac{2}{\mu^{2}} \log \frac{2n}{\varepsilon}\\
&\leq \frac{2n}{\mu^{2}} \log \frac{2s}{\varepsilon} + \frac{2s}{\mu^{2}} \log \frac{2n}{\varepsilon} \ ,
\end{align*}
where we used \eqref{eqn:prec_upper_0} and \eqref{eqn:prec_upper_1}, and take $\Gamma$ small. Note that the total amount of precision increases if the signal magnitude $\mu$ decreases. Combining this result with the bound on the total precision available \eqref{eqn:budget} we can characterize the conditions on $\mu$ for which this procedure fits all the requirements outlined in Section~\ref{sec:setting}.
\begin{proposition} \label{prop:vanilla_upper}
Let $\C$ denote the class of all $s$-sets, and suppose
\begin{equation}\label{eqn:mumu}
\mu \geq \sqrt{\frac{2n}{m} \log \frac{2s}{\varepsilon} + \frac{2s}{m} \log \frac{2n}{\varepsilon}} \ .
\end{equation}
Then the estimator $\hat S$ resulting from the procedure above satisfies $\displaystyle{\max_{S \in \C}} \ \E_S (|\hat{S} \triangle S|) \leq \varepsilon$, and the precision budget of \eqref{eqn:budget}.
\end{proposition}

Since $s \leq n$ the first term on the right hand side of \eqref{eqn:mumu} is always as large as the second term. Thus the scaling of $\mu$ as a function of $n, m, s$ and $\varepsilon$ is determined by the first term. Therefore we have the following corollary.
\begin{corollary}[$s$-sets]\label{cor:vanilla_upper}
Consider the setting of Proposition~\ref{prop:vanilla_upper}. Whenever
$$\mu \geq \sqrt{\frac{4n}{m} \log \frac{2s}{\varepsilon}} \ ,$$
the procedure above produces an estimator $\hat{S}$ satisfying $\displaystyle{\max_{S \in \C}} \ \E_S (|\hat{S} \triangle S|) \leq \varepsilon$, and the precision budget of \eqref{eqn:budget}.
\end{corollary}

We know from \cite{AS_Rui_2012} that, apart from constants, this is the best performance we can hope for when considering the expected Hamming-distance of the estimator (when $s\ll n$). The procedure presented in \cite{malloy:ST} has essentially the same performance as this one, and is also a coordinate-wise method that it is based on sequential thresholding. However, it is parameter adaptive and agnostic about $s$ for a wide range of values.

We now turn our attention to a number of special cases, where the sets belonging to the class $\C$ have some sort of structure. As before, the starting point is some procedure for the noiseless case, specified by $Q_1,Q_2,\ldots,Q_T$. We will make no claim on whether the procedure we define for the noiseless case is optimal in any sense, although in most cases these do give rise to optimal scaling limits. All the noiseless procedures we consider share the common attribute that they consist of two phases. They begin with a \emph{search} phase, where one identifies the general spatial location of the support set. In this phase we sample components of $\vec{x}$ according to some searching method, until we find a certain number $l_1 \leq |S|$ of signal entries. Then we switch to a \emph{refinement} phase, where we exploit the structure of the support set to find a number of entries of $S$. In some cases the proposed procedures alternate between these two phases. Consider the following procedure for the class of $s$-intervals. The search phase simply scans the components until an element of $S$ is found, and the refinement phase explores the coordinates in the neighborhood of the element of $S$ found earlier.

The exact form of queries $Q_1 ,Q_2 ,\ldots$ depends on the specific class under consideration. Likewise, the number of search phases $K$ and how many components to find in each search phase $l_1 ,\dots ,l_K$ depends on the class of possible support sets. In the previous example for the $s$-intervals $K=1$ and $l_1=1$. In what follows we denote the total number of signal entries we wish to find throughout the search phases as $l=\sum_{k=1}^{K} l_k$.

To translate the noiseless-case procedures to the noisy case we  must specify $\alpha_t ,\beta_t$ for each test $\Psi_t ,\ t=1,2,\ldots,T$ to ensure the overall probability of error of the procedure is small. Afterwards we turn our attention to the expected precision used by the procedure. Combining the latter bound with the total amount of precision available \eqref{eqn:budget} we get a condition on the minimal signal strength $\mu$ which is sufficient to ensure the support is recovered accurately. For the control of the overall error probability we can take advantage of the two phases. Suppose we want to keep the probability of error to be less than $\delta$. First, note that since the noiseless case procedure does not sample any coordinate of $\vec{x}$ more than once, we perform at most $n$ tests, thus the conservative choice $\alpha_t \approx \delta /n, \ t=1,2,\ldots,T$ suffices. Now note that throughout the search phases we plan to encounter no more than $l$ non-zero coordinates of $\vec{x}$, so it is reasonable to set $\beta_t \approx \delta /l$. Finally, since there are at most $|S|$ significant components we can observe, in the refinement phase we take $\beta_t \approx \delta /|S|$.

It is crucial to note for a given $t$, $\alpha_t ,\beta_t$ are in general functions of $\{ Q_i ,\Psi_i \}_{i=1}^{t-1}$. This means that when defining the error probabilities we can only use the results of the tests carried out so far, but not the true identity of the entries we sampled. It is important to keep this in mind in the analysis of the procedure. Also, note that the choices above are likely not optimal. For instance the type I error probability will be the same throughout the procedure, and we do not take advantage of the two phases when setting it. Also, for some classes to be considered later on, one can immediately improve the $\alpha_t ,\beta_t$ of the next proposition (e.g. for the $s$-intervals we will perform at most $n/s$ tests in the first phase so setting $\alpha_t = s \delta /n$ for the search phase suffices). Nevertheless, these choices for the probabilities of type I and II errors are simple and general, and yield essentially optimal results.

\begin{proposition} \label{prop:error_prob_upper}
Suppose the noiseless case procedure is of the form described above, and let $\alpha_t = \delta /4n, \ t=1,2,\dots$, $\beta_t = \delta /2l$ for search phase and $\beta_t = \delta /4|S|$ for refinement phase. Then
$$\P_S (\hat{S} \neq S) \leq \delta , \ \forall S \in \C\ .$$
\end{proposition}
\begin{proof}
Consider a noiseless case procedure given by $Q_1,Q_2,\dots$, and any support set $S \in \C$. Let $\mathcal{E}_t$ denote the event that we make an error in the test $\Psi_t$, meaning $\Psi_t \neq \textbf{1}_S (Q_t)$. Let $\overline{\mathcal{E}}_t$ denote the complement of $\mathcal{E}_t$. In what follows we compute the probability that no errors are made.

The support set will be correctly identified if all tests are correct. So clearly
\begin{align*}
\P_S (\hat{S} \neq S) &= 1-\P_S (\hat{S} = S)\\
&\leq 1-\P_S\left(\bigcap_{t=1}^{T} \overline{\mathcal{E}_t} \right)\\
&= 1-\P_S\left(\overline{\mathcal{E}_1} \right)\P_S\left(\left.\overline{\mathcal{E}_2}\right| \overline{\mathcal{E}_1}\right)\cdots\P_S\left(\overline{\mathcal{E}_T}\left| \bigcap_{t=1}^{T-1} \overline{\mathcal{E}_t}\right.\right)\ .
\end{align*}
The above expression upper bounds the probability of error, by considering the case where all the test results coincide with the noiseless case. Since there are at most $n$ zero components being measured in the entire noiseless-case procedure, $l$ significant components being measured in the search phase, and at most $|S|$ significant components being measured in the refinement phase we conclude that
\begin{align*}
\P_S (\hat{S} \neq S) &\leq 1-(1-\delta /4n)^n (1-\delta /2l)^l (1-\delta /4|S|)^{|S|}\\
& \leq n \frac{\delta}{4n} + l \frac{\delta}{2l} + |S| \frac{\delta}{4|S|}\leq \delta\ .
\end{align*}
\end{proof}

This proposition ensures that the noisy case procedure has a probability of error that is sufficiently small. The next step is to evaluate the total expected precision used (considering that the precision $\Gamma$ of each measurement is arbitrarily small). This quantity depends highly on the noiseless case procedure we use for the specific class under consideration. For that reason this calculation is done separately for each case considered.


\subsubsection{$s$-intervals} \label{sec:int}

Consider the class of intervals of length $s$. Formally
\[
\C = \left\{ S \subseteq \{ 1,\dots ,n \} : \ S = \{ i,i+1, \dots ,i+s-1 \} , \ i=1,\dots ,n-s+1 \right\} \ .
\]
For sake of simplicity assume $n/s$ is an integer. This is merely to ease notation in the calculations that follow. The first step is to define a procedure for the noiseless case. Our choice consists of one search and one refinement phase. In the search phase we sample coordinates $1,s+1,2s+1,\dots$, until we find a non-zero coordinate. This gives us the approximate position of the interval. Then we move to the refinement phase to find the left endpoint of the interval by sampling coordinates of $\vec{x}$ to the left of the previously found non-zero coordinate. Note that in the second phase we query at most $s-1$ coordinates.

Formally $Q_t = (t-1)s +1$ for $t=1,\dots T'$, where $T' = \inf \{ t: \ \tilde{Y}_t = 1 \}$, and $Q_t = Q_{t-1} -1$ for $t = T',\dots ,T$, where $T$ is defined in general in \eqref{def:T}. The estimator $\hat{S}$ is defined in \eqref{def:S_hat} as usual. Note that this is an instance of the general procedure described in the setting of Proposition~\ref{prop:error_prob_upper} with $K=1$ and $l_1=l=1$. Taking the corresponding choices for $\alpha_t ,\beta_t , \ t=1,2,\dots$ we ensure that $\P_S (\hat{S} \neq S) \leq \delta , \ \forall S \in \C$. As for the expected precision we can conclude that
\begin{align*}
\E_S \left( \sum_{j} \Gamma_j \right) \leq {} & \E_S \left( \underbrace{\sum_{t=1}^{T'-1} \frac{2}{\mu^{2}} \log \frac{1}{\beta_t}+\frac{2}{\mu^{2}} \log \frac{1}{\min \{ \alpha_t ,\beta_t \}}}_{\text{search}} + \underbrace{\sum_{t=T'+1}^{T} \frac{2}{\mu^{2}} \log \frac{1}{\min \{ \alpha_t ,\beta_t \}}}_{\text{refinement}}  \right)\\
\leq {} & \frac{2}{\mu^2} \left(\frac{n}{s} \log \frac{2}{\delta} + s \log \frac{4n}{\delta} \right)\ .
\end{align*}

Combining this with the bound on the total precision available \eqref{eqn:budget} we get the following.
\begin{proposition} \label{prop:int_upper}
Let $\C$ denote the class of $s$-intervals, and suppose
\[
\mu \geq \sqrt{\frac{2n}{sm} \log \frac{2}{\delta} + \frac{2s}{m} \log \frac{4n}{\delta}} \ .
\]
Then the procedure above results in an estimator $\hat{S}$ satisfying $\displaystyle{\max_{S\in\C}} \ \P_S (\hat{S} \neq S) \leq \delta$ and the precision budget \eqref{eqn:budget}.
\end{proposition}

In addition we can also control the expected Hamming-distance $\E_S (|\hat{S} \triangle S|)$ by recalling \eqref{eqn:hamming_vs_prob}. To guarantee that $\E_S (|\hat{S} \triangle S|)\leq \varepsilon$ we simply have to be slightly more conservative, and require the probability of error $\delta$ to be at most $\varepsilon/s$. A analogous result to that of Proposition~\ref{prop:int_upper} follows immediately. In case signals are sufficiently sparse, meaning $s\ll n$, the first term inside the square root dominates the bound. Therefore we have the following result
\begin{corollary}[$s$-intervals] \label{cor:int_upper}
Consider the setting of Proposition~\ref{prop:int_upper}. Suppose that $s=o\left(\sqrt{n/\log n}\right)$ as $n \to \infty$, and let $\omega_n \to \infty$ be arbitrary.

$(i)$ If
\[
\mu \geq \omega_n \sqrt{\frac{n}{sm}} \ ,
\]
the procedure above gives an estimator $\hat{S}$ such that $\displaystyle{\lim_{n \to \infty} \max_{S \in \C}} \ \P_S (\hat{S} \neq S) = 0$, and that satisfies \eqref{eqn:budget}.

$(ii)$ If
\[
\mu \geq \sqrt{\frac{2n}{sm} \left( \log s + \omega_n \right)} \ ,
\]
the procedure above gives an estimator $\hat{S}$ such that $\displaystyle{\lim_{n \to \infty} \max_{S \in \C}} \ \E_S (|\hat{S} \triangle S|) = 0$, and that satisfies \eqref{eqn:budget}.
\end{corollary}


\subsubsection{Unions of $s$-intervals} \label{sec:u_int}

Now we consider the class whose elements are the union of $k$ disjoint $s$-intervals, where $s$-intervals were defined in the previous subsection. Formally let $\C'$ be the class of $s$-intervals as defined previously. Then
\[
\C = \left\{ S \subseteq \{ 1,\dots ,n \} : \ S= \displaystyle{\bigcup_{i=1}^{k}} S_i, \ S_i \in \C' \ \forall i, \ S_i \cap S_j = \emptyset \ \forall i \neq j \right\} \ .
\]
Again, assume $n/s$ is an integer for simplicity. Note that the cardinality of the support sets belonging to this class is $ks$. In case $s=1$ and $k=s$ this class is the same as the class of $s$-sets considered in Proposition~\ref{prop:vanilla_upper}. When we choose $k=1$ this is the class of $s$-intervals described in Section~\ref{sec:int}. In that sense this class can be viewed as a bridge between the two previous classes.

The procedure for the noiseless case will again consist of one search and one refinement phase. In the search phase we sample coordinates $1,s+1,2s+1,\dots$ until we find $k$ non-zero coordinates. Then in the refinement phase, we sample coordinates to the left of the previously found non-zero coordinates to find the left endpoints of all $k$ intervals. Note that we make at most $k(s-1)$ queries in the second phase.
This procedure is an instance of that described in the setting of Proposition~\ref{prop:error_prob_upper} with $K=1$ and $l_1=l=k$. Taking the corresponding choices for $\alpha_t, \beta_t$ ensures $\P_S (\hat{S} \neq S) \leq \delta , \ \forall S \in \C$. As for the expected precision used we can write
\begin{align*}
\E_S \left( \sum_{j} \Gamma_j \right) \leq {} & \E_S \left( \sum_{t=1}^{T'} \frac{2}{\mu^{2}} \log \frac{1}{\beta_t} + |S| \frac{2}{\mu^{2}} \log \frac{1}{\min\{\alpha_t,\beta_t\}}\right)\\
\leq {} & \frac{2}{\mu^2} \left(\frac{n}{s} \log \frac{2k}{\delta} + ks \log \frac{4n}{\delta} \right)\ .
\end{align*}

Combining this with the bound on the total precision available \eqref{eqn:budget} we arrive to the following.
\begin{proposition} \label{prop:u_int_upper}
Let $\C$ denote the class of unions of $s$-intervals as defined above, and suppose
\[
\mu \geq \sqrt{\frac{2n}{sm} \log \frac{2k}{\delta} + \frac{2ks}{m} \log \frac{4n}{\delta}} \ .
\]
The procedure above results in an estimator $\hat{S}$ satisfying both $\displaystyle{\max_{S\in\C}} \ \P_S (\hat{S} \neq S) \leq \delta$ and the precision budget \eqref{eqn:budget}.
\end{proposition}

In case of sparse signals, that is, when both $s$ and $k$ are small, the first term on the right side dominates this bound. More precisely we have the following.
\begin{corollary}[unions of $s$-intervals] \label{cor:u_int_upper}
Consider the setting of Proposition~\ref{prop:u_int_upper}. Assume $k\geq 2$ and $s\geq 1$ such that $s = o\left(\sqrt{\frac{n\log k}{k \log n}}\right)$ as $n\to\infty$. Let $\omega_n \to \infty$ be arbitrary.

$(i)$ If
\[
\mu \geq \sqrt{\frac{2n}{sm} \left( \log k + \omega_n \right)} \ ,
\]
the procedure above gives an estimator $\hat{S}$  such that $\displaystyle{\lim_{n \to \infty} \max_{S \in \C}} \ \P_S (\hat{S} \neq S) = 0$, and that satisfies \eqref{eqn:budget}.

$(ii)$ If
\[
\mu \geq \sqrt{\frac{2n}{sm} \left( \log ks + \omega_n \right)} \ ,
\]
then the procedure above gives an estimator $\hat{S}$ such that $\displaystyle{\lim_{n \to \infty} \max_{S \in \C}} \ \E_S (|\hat{S} \triangle S|) = 0$, and that satisfies \eqref{eqn:budget}.
\end{corollary}


\subsubsection{$s$-stars} \label{sec:star}

Consider a setting when the coordinates of $\vec{x}$ correspond to edges of a complete undirected graph $G=(V,E)$ with $p$ vertices. We call a support set $S$ an $s$-star if the edges in $G$ corresponding to $S$ form a star in $G$ (see Figure~1 in \cite{Comb_Testing_Lugosi_2010}). Formally, let $e_i \in E$ denote the edge of $G$ corresponding to coordinate $i$ of $\vec{x}$ for all $i=1,\dots ,n$. The class of $s$-stars is defined as
\[
\C = \left\{ S \subseteq \{ 1,\dots ,n \} : \ \displaystyle{\bigcap_{i \in S}} e_i = v \in V , \ |S|=s \right\} \ ,
\]
where $e_i \cap e_j$ is the set of common vertices of edges $e_i ,e_j \in E$. Unlike what was done for the previous classes we use a randomized procedure for the noiseless case. Like was done for $s$-intervals the procedure consists of one search and one refinement phase. In the search phase we randomly search the coordinates of $\vec{x}$ until we find a non-zero coordinate. In the refinement phase we sample the coordinates of $\vec{x}$ which correspond to edges that share a vertex with the non-zero coordinate found in the search phase.

Define $Q_t \sim \mathrm{Unif} \big( \{ 1,\dots ,n \} \setminus \{ Q_1 ,\dots ,Q_{t-1} \} \big)$ for $t=1,\dots ,T'$ where $T' = \inf \{ t: \tilde{Y}_t =1 \}$, and ${Q_t \sim \mathrm{Unif} \big( \tilde{X} \setminus \{ Q_1 ,\dots ,Q_{t-1} \} \big)}$, where $\tilde{X} = \big\{ i \in \{ 1,\dots ,n \} : \ e_i \cap e_{T'} \neq \emptyset \big\}$. The stopping time $T$ and estimator $\hat{S}$ are defined as usual in \eqref{def:T} and \eqref{def:S_hat}. Note that this is an instance of the general procedure described in the setting of Proposition~\ref{prop:error_prob_upper} with $K=1$ and $l_1=1$.

The expected amount of precision used is now a bit more tedious to calculate due to the randomness in the search phase, which in the noisy case is prone to errors. For this reason we slightly modify the procedure above to greatly simplify the analysis. The modification is that in search phase we only take at most $J$ queries. Therefore, if $J$ is small one might end the search phase without finding a star. However, we choose $J$ large enough such that the probability of not querying a signal component is small. If we adjust the error probabilities $\alpha_t ,\beta_t$ accordingly, we can still ensure that the probability of error of the procedure is small. More precisely, we choose $J$ such that $\P_S ( \forall t=1,\dots ,J: \ Q_t \notin S ) \leq \delta /2$. Since
\[
\P_S \left( \forall t=1,\dots ,J: \ Q_t \notin S \right) = \frac{{n-s\choose J}}{{n \choose J}} \leq \left( 1- \frac{s}{n} \right)^J \ ,
\]
choosing $J = n/s \log 2/\delta$ ensures that the probability above is less than $\delta /2$. Now choosing $\alpha_t ,\beta_t$ according to Proposition~\ref{prop:error_prob_upper} with $\delta$ replaced by $\delta /2$ ensures $\P_S (\hat{S} \neq S) \leq \delta , \ \forall S \in \C$.
With this modification the expected amount of precision is bounded by
\begin{align*}
\E_S \left( \sum_{j} \Gamma_j \right) \leq {} & \E_S \left( \sum_{t=1}^{T'} \frac{2}{\mu^{2}} \log \frac{1}{\beta_t} + |S| \frac{2}{\mu^{2}} \log \frac{1}{\alpha_t} + \sum_{t=T'+1}^{T} \frac{2}{\mu^{2}} \log \frac{1}{\beta_t} \right)\\
\leq {} & \frac{2}{\mu^2} \left( J \log \frac{4}{\delta} + |S| \log \frac{8n}{\delta} + 2(p-2) \log \frac{8s}{\delta} \right) \\
\leq {} & \frac{2}{\mu^2} \left( \frac{n}{s} \left( \log \frac{4}{\delta} \right)^2 + s \log \frac{8n}{\delta} + \sqrt{8n} \log \frac{8s}{\delta} \right)\ .
\end{align*}
Combining this with the bound on the total precision available \eqref{eqn:budget} we get the following.
\begin{proposition} \label{prop:star_upper}
Let $\C$ be the class of $s$-stars as defined above and suppose
\[
\mu \geq \sqrt{\frac{2n}{sm} \left( \log \frac{4}{\delta} \right)^2 + \frac{2s}{m} \log \frac{8n}{\delta} + \frac{\sqrt{32n}}{m} \log \frac{8s}{\delta}} \ .
\]
The procedure above results in an estimator $\hat{S}$ satisfying both $\displaystyle{\max_{S\in\C}} \ \P_S (\hat{S} \neq S) \leq \delta$ and the precision budget \eqref{eqn:budget}.
\end{proposition}

In case of sparse signals the first term on the right side dominates this bound. More precisely we have the following.
\begin{corollary} [$s$-stars]\label{cor:star_upper}
Consider the setting of Proposition~\ref{prop:star_upper}. Suppose $n \to \infty$ such that $s = o(\sqrt{n}/\log n)$. Let $\omega_n \to \infty$  be arbitrary.

$(i)$ If
\[
\mu \geq \omega_n \sqrt{\frac{n}{sm}} \ ,
\]
the procedure above gives an estimator $\hat{S}$ such that $\displaystyle{\lim_{n \to \infty} \max_{S \in \C}} \ \P_S (\hat{S} \neq S) = 0$, and that satisfies \eqref{eqn:budget}.

$(ii)$ If
\[
\mu \geq \sqrt{\frac{2n}{sm} \left( \log s + \omega_n \right)^2} \ ,
\]
the procedure above gives an estimator $\hat{S}$ such that $\displaystyle{\lim_{n \to \infty} \max_{S \in \C}} \ \E_S (|\hat{S} \triangle S|) = 0$, and that satisfies \eqref{eqn:budget}.
\end{corollary}


\subsubsection{Unions of $s$-stars} \label{sec:u_star}

The unions of $k$ non-intersecting $s$-stars is a generalization of the class of $s$-stars defined in the previous section. Suppose for technical reasons that $k < s$. Let $\C'$ be the class of $s$-stars as defined previously. Then
\[
\C = \left\{ S \subseteq \{ 1,\dots ,n \} : \ S= \displaystyle{\bigcup_{i=1}^{k}} S_i, \ S_i \in \C' \ \forall i, \ S_i \cap S_j = \emptyset \ \forall i \neq j \right\} \ .
\]
Note that the cardinality of the support sets belonging to this class is $ks$. In contrast to what we have done before, the proposed noiseless procedure will consist of alternating search and refinement phases. In the search phases we randomly search coordinates of $\vec{x}$ until we find a signal coordinate. Then we switch to the a refinement phase and sample every coordinate of $\vec{x}$ which correspond to edges that share a vertex with the non-zero coordinate found previously. After doing so it may happen that we find signal entries corresponding to more than one star, since the stars may share vertices. If there are stars left partly explored, we continue sampling edges that possibly belong to not yet fully explored stars. If there are no partly explored stars, we switch back to the search phase. We keep iterating until we found all $k$ stars of the graph.

Formally $Q_t \sim \mathrm{Unif} \big( \{ 1,\dots ,n \} \setminus \{ Q_1 ,\dots ,Q_{t-1} \} \big)$ in the search phases. Let $\tilde{X}_t$ denote the set of edges that can belong to partly explored stars up to time $t$. Then the queries of the subsequent search phase can be defined as $Q_t \sim \mathrm{Unif} \big( \tilde{X}_t \setminus \{ Q_1 ,\dots ,Q_{t-1} \} \big)$. Note that this stills fits the setting of Proposition~\ref{prop:error_prob_upper} with $K \leq k$ being random and $l_1 = l_2 = \dots = l_K =1$.

Analogously to what was done for $s$-stars  we consider a simple modification to facilitate the analysis: each time we are in a search phase we take at most $J$ queries. We choose $J$ such that the noiseless case procedure fails with small probability. Note that we perform at most $k$ search phases, and in each of them there are at least $s$ unexplored signal components. Thus, using essentially the same calculation as before, we get that by choosing $J = n/s \log 2k/\delta$ we ensure that the probability of not querying a signal coordinate in any of the search phases is at most $\delta/2$. Finally, choosing $\alpha_t ,\beta_t$ according to \eqref{prop:error_prob_upper} with $\delta$ replaced by $\delta /2$ yields $\P_S (\hat{S} \neq S) \leq \delta , \ \forall S \in \C$.

Note that the number of queries we perform in all of the search and refinement phases is at most $kJ$ and $2kp$, respectively. However, for the expected number of queries performed throughout the search phases we can get a slightly better upper bound, which is necessary to get a more accurate dependence on the parameter $k$. Recall that $\mathcal{E}_t$ denotes the event that we make an error in the test $\Psi_t$, i.e. $\Psi_t \neq \textbf{1}_S (Q_t), \ t=1,2,\dots$. Also, let $\mathcal{E}_0$ denote the event that there is a search phase in which we do not query any coordinate containing a signal, and $T_A$ denote the number of queries in search phases. Finally, let the number of queries in the $j$th search phase be $T_A^{(j)}$. Using the mean of the negative hypergeometric distribution, we have $\E_S \left(T_A^{(j)}\left| \cap_{t=0}^{T} \ \overline{\mathcal{E}}_t \right.\right) \leq n/ (s k_j)$, where $k_j$ is the number of unexplored stars in search phase $j$. Noting that $k_1 = k$ and $k_{j+1} < k_j, \ j=1,\dots ,K$ we have the following bound
\[
\E_S \left( T_A \left| \bigcap_{t=0}^{T} \ \overline{\mathcal{E}}_t \right. \right) = \sum_{j=1}^{K} \ \E_S \left( T_A^{(j)} \left| \bigcap_{t=0}^{T} \ \overline{\mathcal{E}}_t \right. \right) \leq \sum_{j=1}^{k} \frac{n}{js} \leq \frac{n}{s} (\log k +1) \ .
\]

Finally, through the law of total expectation we get
\[
\E_S \left( T_A \right) \leq \frac{n}{s} (\log k +1) + \delta kJ \ .
\]

We are now ready to compute a bound on the precision used by procedure, which is given by
\begin{align*}
\E_S \left( \sum_{j} \Gamma_j \right) \leq {} & \E_S \left( T_A \frac{2}{\mu^{2}} \log \frac{4k}{\delta} + |S| \frac{2}{\mu^{2}} \log \frac{8n}{\delta} + 2k(p-2) \frac{2}{\mu^{2}} \log \frac{8ks}{\delta} \right) \\
\leq {} & \frac{2}{\mu^2} \left( \left( \frac{n}{s} (\log k +1) + \delta kJ \right) \log \frac{4k}{\delta} + ks \log \frac{8n}{\delta} + k\sqrt{8n} \log \frac{8ks}{\delta} \right) \\
\leq {} & \frac{2}{\mu^2} \left(\frac{n(1+ \delta k)}{s} \left( \log \frac{4k}{\delta} \right)^2 + ks \log \frac{8n}{\delta} + k\sqrt{8n} \log \frac{8ks}{\delta} \right)\ .
\end{align*}

Combining this with the bound on the total precision available \eqref{eqn:budget} we get the following.
\begin{proposition} \label{prop:u_star_upper}
Let $\C$ be the class of unions of $k$ disjoint $s$-stars as defined above and suppose
\[
\mu \geq \sqrt{\frac{2n (1+ \delta k)}{sm} \left( \log \frac{4k}{\delta} \right)^2 + \frac{2ks}{m} \log \frac{8n}{\delta} + \frac{k\sqrt{32n}}{m} \log \frac{8ks}{\delta}} \ .
\]
The procedure above results in an estimator $\hat{S}$ satisfying both $\displaystyle{\max_{S\in\C}} \ \P_S (\hat{S} \neq S) \leq \delta$ and the precision budget \eqref{eqn:budget}.
\end{proposition}

The result of the above proposition is perhaps a bit difficult to digest, but provided $s$ and $k$ are small relative to $n$ the first term in the right side dominates the bound.

\begin{corollary}[unions of $s$-stars] \label{cor:u_star_upper}
Consider the setting of Proposition~\ref{prop:u_star_upper}. Suppose $s,k \to \infty$ such that $s=o\left(\frac{\sqrt{n\log k}}{k\log n}\right)$.

$(i)$ If
\[
\mu \geq \sqrt{\frac{4n}{sm} \log^2 4k} \ ,
\]
the procedure above gives an estimator $\hat{S}$ such that $\displaystyle{\lim_{n \to \infty} \max_{S \in \C}} \ \P_S (\hat{S} \neq S) = 0$, and that satisfies \eqref{eqn:budget}.

$(ii)$ If
\[
\mu \geq \sqrt{\frac{4n}{sm} \log^2 ks  }\ ,
\]
the procedure above gives an estimator $\hat{S}$ such that $\displaystyle{\lim_{n \to \infty} \max_{S \in \C}} \ \E_S (| \hat{S} \triangle S|) = 0$, and that satisfies \eqref{eqn:budget}.
\end{corollary}


\subsubsection{$s$-submatrices} \label{sec:submat}

In this setting the entries of $\vec{x}$ are identified with the elements a matrix $M \in \R^{n_1 \times n_2}$ (let the number of elements in the matrix be $n=n_1n_2$). We assume the support set $S$ is a subset of $\{1,\ldots,n_1\}\times\{1,\ldots,n_2\}$ and furthermore we assume it corresponds to a submatrix of size $s$. Formally, the class of all $s$-submatrices is defined as
\[
\C = \left\{ S_1\times S_2: S_1\subseteq\{1,\ldots,n_1\}, S_2\subseteq\{1,\ldots,n_2\},\text{ and } |S_1| \cdot |S_2| =s\right\} \ ,
\]
where $S_1 \times S_2$ denotes the cartesian product of $S_1$ and $S_2$. Note that if either $n_1$ or $n_2$ is of the same order as $n$, then this setting becomes similar to the unstructured case, but if $n_1,n_2\approx\sqrt{n}$ there is a significant amount of structure on can take advantage of. Consider the following simple noiseless support recovery procedure: in a first phase randomly search the coordinates of $\vec{x}$ to find a non-zero coordinate. Once such a coordinate is found explore coordinates of $\vec{x}$ corresponding to the row and column of the non-zero coordinate found in the previous phase. Clearly, this fits the general procedure described in the setting of Proposition~\ref{prop:error_prob_upper}, with $K=1$ and $l_1=l=1$. Like in the case of $s$-stars we stop the random search in the first phase after $J=(n/s)\log(2/\delta)$ queries to facilitate the analysis. For the expected amount of precision used we have
\begin{align*}
\E_S \left( \sum_{j} \Gamma_j \right) \leq {} & \frac{2}{\mu^2} \left( J \log \frac{4}{\delta} + s \log \frac{8n}{\delta} + (n_1 + n_2) \log \frac{8s}{\delta} \right)\\
\leq {} & \frac{2}{\mu^2} \left(\frac{n}{s} \left( \log \frac{4}{\delta} \right)^2 + s \log \frac{8n}{\delta} + (n_1 + n_2) \log \frac{8s}{\delta} \right)\ .
\end{align*}

\begin{proposition} \label{prop:submat_upper}
Let $\C$ denote the class of submatrices as defined above with and suppose we have
\[
\mu \geq \sqrt{\frac{2n}{sm} \left( \log \frac{4}{\delta} \right)^2 + \frac{2s}{m} \log \frac{8n}{\delta} + \frac{2(n_1 + n_2)}{m} \log \frac{8s}{\delta}} \ .
\]
The procedure above results in an estimator $\hat{S}$ satisfying both $\displaystyle{\max_{S\in\C}} \ \P_S (\hat{S} \neq S) \leq \delta$ and the precision budget \eqref{eqn:budget}.
\end{proposition}

In case of sparse signals, that is when $s \ll n$, and both $n_1 \approx \sqrt{n}$ and $n_2 \approx \sqrt{n}$ the first term on the right side dominates this bound. When $\max \{ n_1 ,n_2 \}$ is at the order of $n$, the situation becomes similar to the unstructured case, and the third term dominates the bound (so one recovers essentially the result in Corollary~\ref{cor:vanilla_upper}). Concerning the former case one has the following result.

\begin{corollary}[$s$-submatrices] \label{cor:submat_upper}
Consider the setting of Proposition~\ref{prop:submat_upper}.  Assume $n_1=n_2=\sqrt{n}$ and $s=o(\sqrt{n}/\log n)$ as $n\to\infty$. Let $\omega_n \to \infty$ be arbitrary. 

$(i)$ If
\[
\mu \geq \omega_n \sqrt{\frac{n}{sm}} \ ,
\]
the procedure above gives an estimator $\hat{S}$ such that $\displaystyle{\lim_{n \to \infty} \max_{S \in \C}} \ \P_S (\hat{S} \neq S) = 0$, and that satisfies \eqref{eqn:budget}.

$(ii)$ If
\[
\mu \geq \sqrt{\frac{2n}{sm} \left( \log s + \omega_n \right)^2} \ ,
\]
the procedure above gives an estimator $\hat{S}$ such that $\displaystyle{\lim_{n \to \infty} \max_{S \in \C}} \ \E_S (|\hat{S} \triangle S|) = 0$, and that satisfies \eqref{eqn:budget}.
\end{corollary}

\begin{remark}
In all settings considered one can get exactly the same results with some degree of adaptivity to sparsity level, characterized by $s$ and $k$. For instance, if one considers the class of all $k$ and $k-1$ unions of $s$-intervals or $s$-stars the results of Corollaries~\ref{cor:u_int_upper} and \ref{cor:u_star_upper} still hold. Likewise mild adaptivity to $s$ is also possible. Furthermore, all the results above will hold if the empty set if added to the class $\C$ under consideration.
\end{remark}



\section{Lower Bounds} \label{sec:lower_bounds}

In this section we derive bounds for the signal strength $\mu$ for each special case considered earlier, such that if $\mu$ falls below these bounds, reliably recovering the support set $S \in \C$ is impossible. First we derive bounds for non-adaptive sensing for comparison purposes. For the non-adaptive case we derive lower bounds considering the error metric $\P_S (\hat{S} \neq S)$. These are lower bounds considering the error metric $\E_S (|\hat{S} \triangle S|)$ as well, since the latter dominates the former. The bounds we present for the non-adaptive case are not sharp, particularly when the signal is not sparse. Nonetheless in the sparse setting they capture the essence of the difficulty of support recovery, and illustrate well the gains one might achieve by using adaptive sensing procedures.

Second, we derive lower bounds for adaptive sensing, which will show the near-optimality of the proposed procedure for the classes considered. In this case we are considering the Hamming-distance as a metric for evaluating the performance of an estimator. It is not yet clear if similar methods can be used to construct lower bounds when considering $\P_S (\hat{S} \neq S)$ as an error metric in general. The proofs of the results of this section make use of tools from \cite{AS_Rui_2012} and \cite{Tsybakov_2009}.


\subsection{Non-Adaptive Sensing} \label{sec:non_adapt_lower}

In this subsection we consider non-adaptive sensing for support recovery. The problem setting is the same as in Section~\ref{sec:setting}, the only difference being that in the non-adaptive setting we have to specify $\{ A_j, \Gamma_j \}_{j=1,2,\dots}$ before any observations are made. All the bounds presented here are based on Proposition 2.3 in in \cite{Tsybakov_2009}, which states
\begin{lemma} \label{lemma:tsybakov}
Let $\P_0, \dots, \P_M$ be probability measures on $(\mathcal{X},\mathcal{A})$ and let $\Psi: \mathcal{X} \to \{ 0, \dots ,M \}$ be any $\mathcal{A}$-measurable function. If
\[
\frac{1}{M} \displaystyle{\sum_{j=1}^{M}} D( \P_j \| \P_0 ) \leq a
\]
then
\[
\displaystyle{\max_{j=0, \dots ,M}} \P_j (\Psi \neq j) \geq \displaystyle{\sup_{0< \tau <1}} \left( \frac{\tau M}{1+ \tau M} \left( 1+ \frac{a + \sqrt{a/2}}{\log \tau} \right) \right) \ .
\]
\end{lemma}

We can use this result directly to get general lower bounds for $\mu$ in the non-adaptive setting. First let $\P_0, \dots ,\P_M$ be the probability measures induced by the sampling $\vec{x}$ with parameters $\{ A_j , \Gamma_j \}_{j=1,2,\dots}$, when the support sets are $S_0, \dots ,S_M$ respectively, where $S_i \in \C$. We take $S_0, \dots ,S_M$ to be all the support sets in $\C$, so that $M = |\C|-1$. For fixed $S_k ,S_l , \ k \neq l$ we have $D( \P_k \| \P_l ) = \mu^{2} /2 \sum_{j: A_j \in S_k \triangle S_l} \Gamma_j$. Let us define $b_i = \sum_{j: A_j =i} \Gamma_j$. Then
\[
\displaystyle{\sum_{j=1}^{M}} D( \P_j \| \P_0 ) = \frac{\mu^{2}}{2} \displaystyle{\sum_{S' \in \C \setminus \{ S_0 \}} \ \sum_{i \in S_0 \triangle S'}} \ b_i \ .
\]

We need to evaluate the quantity above. Since we can choose any $S_0 \in \C$ we can consider the set which makes the hypothesis test harder. On the other hand, the measurement budget constraint \eqref{eqn:budget} implies that $\sum_i b_i \leq m$. This yields the following optimization problem
\[
\displaystyle{\max_{b \in \R_{+,0}^{n}: \ ||b||_1 \leq m} \ \min_{S \in \C} \ \sum_{S' \in \C \setminus \{ S \}} \ \sum_{i \in S \triangle S'}} \ b_i \ .
\]
where $b = (b_1, \dots ,b_n)^T$. The solution of this problem can be found explicitly if the class $\C$ under consideration has the following symmetry property (as introduced in \cite{AS_Rui_2012}).

\begin{definition} \label{def:symm_class}
Let $S \in \C$ be drawn uniformly at random. If for all $i \in \{ 1, \dots ,n \}$ we have $\P (i \in S) = s/n$ the class $\C$ is symmetric.
\end{definition}

We have the following proposition, proved in the Appendix.
\begin{proposition} \label{prop:prec_symm_classes}
Suppose $\C$ is symmetric. Then
\[
\displaystyle{\max_{b \in \R_{+,0}^{n}: \ ||b||_1 \leq m} \ \min_{S \in \C} \ \sum_{S' \in \C \setminus \{ S \}} \ \sum_{i \in S \triangle S'}} \ b_i
\]
is attained when $b_i = m/n, \ i= 1, \dots ,n$.
\end{proposition}

We are now in position to prove the proposition which we can use to get lower bounds for $\mu$ in our special cases.

\begin{proposition} \label{prop:non_adapt_lower}
Let $\C$ be symmetric and suppose $1+ \sqrt{2} \leq (1- 2\varepsilon) \log (|\C| -1)$. If
\[
\mu^{2} \leq (1- 2\varepsilon) \frac{n}{2|S|m} \log (|\C| -1) \ ,
\]
then no non-adaptive procedure can have
\[
\P_S (\hat{S} \neq S) \leq \varepsilon , \ \forall S \in \C \ .
\]
\end{proposition}
\begin{proof}
Let $\P_0, \dots ,\P_M$ be the probability measures induced by the sampling $\vec{x}$ with parameters $\{ A_j , \Gamma_j \}_{j=1,2,\dots}$, when the support sets are $S_0, \dots ,S_M$ respectively, where $S_i \in \C$. We take $S_0, \dots ,S_M$ to be all the support sets in $\C$, that is $M = |\C|-1$. By Proposition~\ref{prop:prec_symm_classes} we know $b_i = m/n , \ i=1,2,\dots$ is the optimal choice for distributing the precision in the non-adaptive setting for symmetric $\C$. From this we have
\[
\frac{1}{M} \displaystyle{\sum_{j=1}^{M}} D( \P_j \| \P_0 ) \leq \displaystyle{\max_{j=1, \dots ,M}} D( \P_j \| \P_0 ) = |S| \frac{m}{n} \mu^{2} := a \ .
\]
Then by Lemma \ref{lemma:tsybakov} we have
\[
\displaystyle{\sup_{S \in \C}} \ \P_S(\hat{S} \neq S) \geq \displaystyle{\sup_{0< \tau <1}} \left( \frac{\tau M}{1+ \tau M} \left( 1+ \frac{a + \sqrt{a/2}}{\log \tau} \right) \right) \ .
\]
Setting $\tau = 1/M$ we get
\[
\displaystyle{\sup_{S \in \C}} \ \P_S(\hat{S} \neq S) \geq \frac{1}{2} \left( 1- \frac{a + \sqrt{a/2}}{\log M} \right) \ .
\]

The right side of the above expression is bounded below by $\varepsilon$ whenever
\begin{equation} \label{eqn:condition}
a + \sqrt{a/2} \leq (1- 2\varepsilon) \log M
\end{equation}
Plugging the values of $a$ and $M$ into the above inequality immediately yields bounds for $\mu$. However, to make the bound more transparent, assume $\C$ is such that $1+ \sqrt{2} \leq (1- 2\varepsilon) \log (|\C| -1)$. Then every $a$ satisfying
\[
2a \leq (1- 2\varepsilon) \log M
\]
also satisfies \eqref{eqn:condition}. The statement now follows.
\end{proof}

Note that the condition $1+ \sqrt{2} \leq (1- 2\varepsilon) \log (|\C| -1)$ is not necessary to get the bound for $\mu$, its role is merely to make the bound more transparent. Furthermore, it simply requires $\C$ to be large enough compared to $\varepsilon$. Since we are interested in cases where $\C$ is large and $\varepsilon$ is small we can always safely assume this condition holds provided $\varepsilon$ is small enough. The result of Proposition~\ref{prop:non_adapt_lower} is remarkably simple, as the lower bound depends exclusively on the cardinality of the class under consideration. With this in hand it is immediate to get non-adaptive lower bounds for all the classes considered in the paper.

\begin{proposition} \label{prop:non-adapt_lower_2}
A necessary condition to ensure that any non-adaptive procedure satisfies $\P_S (\hat{S} \neq S) \leq \varepsilon , \ \forall S \in \C$ is given by the following expressions, for the different classes $\C$:
\begin{itemize}
\item $s$-sets: $\mu \geq \sqrt{(1- 2\varepsilon) \frac{n}{2sm} \log \left({n \choose s} -1\right)} \ .$

\item $s$-intervals: $\mu \geq \sqrt{(1- 2\varepsilon) \frac{n}{2sm} \log \left(\frac{n}{s} -1\right)} \ .$

\item unions of $k$ disjoint $s$-intervals: $\mu \geq \sqrt{(1- 2\varepsilon) \frac{n}{2ksm} \log \left({n/s \choose k} -1\right)} \ .$

\item unions of $k$ disjoint $s$-stars\footnote{assuming $k(s+1)\leq p$.}: $\mu \geq \sqrt{(1- 2\varepsilon) \frac{n}{2ksm} \log \left({p \choose k(s+1)} -1\right)}\ .$

\item $s$-submatrices: $\mu \geq \sqrt{(1- 2\varepsilon) \frac{n}{2sm} \log \left({n_1 \choose \sqrt{s}} {n_2 \choose \sqrt{s}} -1\right)} \ .$
\end{itemize}
\end{proposition}

\begin{proof}
The case of $s$-sets is straightforward from Proposition~\ref{prop:non_adapt_lower}. The class of $s$-intervals is not symmetric, however, its subclass
$$\left\{\{1,\ldots,s\},\{s+1,\ldots,2s\},\ldots,\{n-s+1,\ldots,n\}\right\}$$
is, therefore we can apply Proposition~\ref{prop:non_adapt_lower} for this subclass. For the class of unions of intervals, we consider a similarly constructed subclass to get the bound above. In case of the unions of stars, we can consider the subclass of stars with distinct vertices. The size of this subclass is lower bounded by ${p \choose k(s+1)}$. For the submatrices, we consider the subclass of submatrices of size $\sqrt{s} \times \sqrt{s}$.
\end{proof}

Using the previous results we can state the following corollary considering the large $n$ behavior of the non-adaptive lower bounds.
\begin{corollary} \label{cor:non-adapt_lower}
In order to have $\displaystyle{\lim_{n \to \infty}} \ \P_S (\hat{S} \neq S) = 0$ for $n\to\infty$ any non-adaptive procedure must satisfy, for some constant $c > 0$
\begin{itemize}
\item $s$-sets: $\mu \geq c \sqrt{\frac{n}{2m} \log \frac{n}{s}} \ ,$ when $s = o(n)$.

\item $s$-intervals: $\mu \geq c \sqrt{\frac{n}{2sm} \log \frac{n}{s}} \ ,$ when $s = o(n)$.

\item unions of $k$ disjoint $s$-intervals: $\mu \geq c \sqrt{\frac{n}{2sm} \log \frac{n}{ks}} \ ,$ when $ks = o(n)$.

\item unions of $k$ disjoint $s$-stars: $\mu \geq c \sqrt{\frac{n}{2m} \log \frac{\sqrt{2n}}{ks}} \ ,$ when $ks = o(\sqrt{n})$.

\item $s$-submatrices: $\mu \geq c \sqrt{\frac{n}{4 \sqrt{s} m} \log \frac{n}{s}} \ ,$ when $n_1 = n_2 = \sqrt{n}$, $s_1 = s_2 = \sqrt{s}$, and $s = o(n)$.
\end{itemize}
\end{corollary}

The previous results sheds some light on the limits of support recovery in the non-adaptive setting. If the size of the support set ($s$ or $ks$) is sufficiently small relative to $n$ the $\log n$ factor is unavoidable for non-adaptive support recovery. On the contrary, this factor does not appear in the performance bounds in any of the cases we consider. For the class of unions of intervals, a factor of $\sqrt{1/s}$ appears in the lower bounds above, which means it might be possible to capitalize on the structure in the non-adaptive case as well (and this is indeed the case). For the unions of stars, however, this is no longer true. In fact this class is so large that non-adaptive procedures can no longer take significant advantage from the structure of the support sets. This is in stark contrast with what is possible with adaptive sensing (see Corollary~\ref{cor:star_upper}). Similar remarks apply to the class of submatrices as well.


\subsection{Adaptive Sensing} \label{seq:adaptive_lower}

In this section we derive lower bounds for the signal strength $\mu$ in the adaptive sensing setting. We measure the performance of an estimator by the expected Hamming-distance $\E_S (|\hat{S} \triangle S|)$. In some cases we are also able to prove lower bounds for the error metric $\P_S (\hat{S} \neq S)$. Comparing the bounds of this section with the performance bounds of Section~\ref{sec:performance} shows the near optimality of the proposed procedure for sparse signals for all the classes considered.


\subsubsection{$s$-sets (unstructured case)}

This case is considered in \cite{AS_Rui_2012} and lower bounds are shown for a slightly larger class, which consists of all $s$, $s-1$ and $s+1$ sets. However, it turns out that a similar result holds also if one considers the class of all $s$ and $(s-1)$-sets only. Let $\C'$ denote this class, and suppose there is a sensing procedure and estimator $\hat S$ for which
\[
\max_{S\in\C'} \ \E_S (|\hat{S} \triangle S|) \leq \varepsilon \ .
\]
Lemma 4.1 in \cite{AS_Rui_2012} shows that it suffices to consider only \emph{symmetric} estimators, which satisfy
\[
\forall i,j \in S: \ \P_S (i \notin \hat{S}) = \P_S ( j \notin \hat{S}) \ ,
\]
and
\[
\forall i,j \notin S: \ \P_S ( i \notin \hat{S}) = \P_S ( j \notin \hat{S}) \ ,
\]
for any $S \in \C'$. This follows since any estimator $\hat S$ can be symmetrized without affecting their worst case performance when the class under consideration is closed under permutations. It is easily shown that, for symmetric estimators we have
\[
\forall i,j \in S: \ \E_S \left(\sum_{k:A_k = i} \Gamma_k \right) = \E_S \left(\sum_{k:A_k = j} \Gamma_k \right) \ ,
\]
and
\[
\forall i,j \notin S: \ \E_S \left(\sum_{k:A_k = i} \Gamma_k \right) = \E_S \left(\sum_{k:A_k = j} \Gamma_k \right) \ ,
\]
for any $S \in \C'$. We can then proceed as follows. Let $S\in\C'$ and $i\in\{1,\ldots,n\}$ be arbitrary, and such that $|S|=s-1$ and $i\notin S$. Define also $S'=S \cup \{i\}$. For the event $\{ i \notin \hat{S} \}$ we have
\[
D(\P_S \| \P_{S'}) \geq -\log\left(2\P_S(i \notin \hat{S})+ 2\P_{S'}(i = \hat{S})\right)\ .
\]
Using the symmetry of the estimator we can easily bound the left hand side as follows.
\[
D(\P_S \| \P_{S'}) = \frac{\mu^2}{2} \E_S \left(\sum_{j:A_j = i} \Gamma_j \right) \leq \frac{\mu^2}{2} \frac{m}{n-s+1} \ .
\]
Furthermore, also by symmetry
\[
\P_S (i\in\hat S) \leq \varepsilon /(n-s+1), \ \P_{S'} (i\notin \hat S) \leq \varepsilon /s \ ,
\]
whenever we have $\E_S (|\hat{S} \neq S|) \leq \varepsilon$. Putting everything together yields the following proposition.
\begin{proposition} \label{prop:vanilla_lower}
Let $\C'$ denote the class of all subsets of $\{1,\dots ,n \}$ with cardinality either $s-1$ or $s$. Suppose that $\displaystyle{\max_{S \in \C'}} \ \E_S (|\hat{S} \triangle S|) \leq \varepsilon$. Necessarily
\[
\mu \geq \sqrt{\frac{2(n-s)}{m} \left( \log s + \log\frac{n-s}{n+1} + \log\frac{1}{2\varepsilon} \right)}\ .
\]
\end{proposition}

In the large $n$ regime and when considering sparse signals we have the following result.
\begin{corollary}[$s$-sets] \label{cor:vanilla_lower}
Consider the setting of Proposition~\ref{prop:vanilla_lower}, and suppose $s=o(n)$ as $n \to \infty$. If there is an adaptive sensing and estimation strategy such that $\displaystyle{\lim_{n \to \infty} \max_{S\in\C'}} \ \E_S (|\hat{S} \triangle S|) = 0$ then necessarily
\[
\mu \geq \sqrt{\frac{2n}{m} (\log s+\omega_n)}\ ,
\]
where $\omega_n\to\infty$.
\end{corollary}
This shows that, in the asymptotic regime, our adaptive-sensing procedure is near optimal in the unstructured case, when the signal is sparse.


\subsubsection{$s$-intervals}

In the case of $s$-intervals we can get a lower bound for the probability of error as the metric of interest. This lower bound is, however, a bit loose on the dependence on $\varepsilon$.
\begin{proposition} \label{prop:int_lower_v2}
Let $\C$ be the class of $s$-intervals. Assume $\displaystyle{\max_{S \in \C}} \ \P_S (\hat{S} \neq S) \leq \varepsilon$, where $\varepsilon>0$. Then necessarily
\[
\mu \geq (1- \varepsilon) \sqrt{\frac{n}{2sm}} \ .
\]
\end{proposition}
\begin{proof}
Assume without loss of generality that $n/s$ is an even integer. Define the class of disjoint intervals
\begin{equation}\label{eqn:disjoint_int}
\D=\left\{\{1,\ldots,s\},\{s+1,\ldots,2s\},\ldots,\{n-s+1,\ldots,n\}\right\}\ .
\end{equation}
Since $\D\subseteq\C$ it suffices to show the lower bound for this smaller class. Now partition the class in two disjoint sets of the same size, namely $\D_1$ and $\D_2$, such that $\D_1\cup\D_2=\D$ and $\D_1\cap\D_2=\emptyset$. To show the lower bound we consider a test of two simple hypothesis. Namely, under $H_1$ assume $S\sim\mathrm{Unif}(\D_1)$ and under $H_2$ assume $S\sim\mathrm{Unif}(\D_2)$. In words, data under each hypothesis is generated by first selecting the support set $S$ from either distribution, and then collecting data $D=\{Y_j,A_j,\Gamma_j\}_{1,2,\ldots}$ under the model indexed by $S$. Therefore this is a test between two simple hypothesis.

Given a support estimator $\hat S$ one can construct a test
$$\Phi(D)=\left\{\begin{array}{ll}
1 & \text{ if } \hat S\in\D_1\\
2 & \text{ otherwise}
\end{array}\right. \ .$$ 
Clearly, if $\P_S (\hat{S} \neq S) \leq \varepsilon \ \forall S \in \D$, then $\P_1 (\Phi(D)=2) + \P_2 (\Phi(D)=1)\leq\varepsilon$, where $\P_i$ denotes the distribution of $D=\{ Y_j ,A_j ,\Gamma_j \}_{1,2,\dots}$ under $H_i$. Let $\P_0$ denote the distribution of $D$ when $S= \emptyset$ and $TV(\cdot,\cdot)$ denote the total-variation distance. We have
\begin{align}
\P_1 (\Phi(D)=2) + \P_2 (\Phi(D)=1) & \geq 1- TV(\P_1 ,\P_2 )\nonumber\\
&\geq 1- \big( TV(\P_0 ,\P_1 ) + TV(\P_0, \P_2 ) \big) \nonumber\\
& = 1- 2 TV(\P_0 ,\P_1 ) \nonumber\\
&\geq  1- \sqrt{2D(\P_0 \| \P_1)} \label{eqn:double_kl}\ ,
\end{align}
where the second inequality follows from the triangle inequality, and the third inequality follows from the first Pinsker inequality \cite{Tsybakov_2009}.

The Kullback-Leibler between $\P_0$ and $\P_1$ can be bounded as follows.
\begin{align*}
D(\P_0 \| \P_1) & =  \sum_j \E_0 \left[ \log \frac{\d\P_0(Y_j|A_j,\Gamma_j)}{\d\P_1(Y_j|A_j,\Gamma_j)} \right]\\
& = -\sum_j \E_0 \left[ \log \frac{\d\P_1(Y_j|A_j,\Gamma_j)}{\d\P_0(Y_j|A_j,\Gamma_j)} \right]\\
& = -\sum_j \E_0 \left[ \log \frac{\frac{1}{|\D_1|}\sum_{S\in\D_1}\d\P_S(Y_j|A_j,\Gamma_j,S)}{\d\P_0(Y_j|A_j,\Gamma_j)} \right]\\
& = -\sum_j \E_0 \left[ \log \frac{1}{|\D_1|}\sum_{S\in\D_1}  \exp \left( -\frac{\Gamma_j}{2} \mu \textbf{1}\{ A_j = S \} (\mu - 2 Y_j) \right) \right]\\
& \leq -\sum_j \E_0 \left[ \frac{1}{|\D_1|}\sum_{S\in\D_1}  -\frac{\Gamma_j}{2} \mu \textbf{1}\{ A_j = S \} (\mu - 2 Y_j) \right]\ ,
\end{align*}
where the last step follows from Jensen's inequality. Therefore
\begin{align*}
D(\P_0 \| \P_1) & \leq  \frac{1}{|\D_1|}\sum_{S\in\D_1} \E_0 \left[ \sum_j \frac{\Gamma_j}{2} \mu \textbf{1}\{ A_j = S \} (\mu - 2 Y_j) \right]\\
& =  \frac{1}{|\D_1|}\sum_{S\in\D_1} \E_0 \left[ \sum_j \frac{\Gamma_j}{2} \mu \textbf{1}\{ A_j = S \} (\mu - 2 Y_j) \right]\\
& =  \frac{1}{|\D_1|}\sum_{S\in\D_1} \E_0 \left[ \sum_{j:A_j =S} \frac{\Gamma_j}{2} \mu^2 \right]\\
& \leq  \frac{\mu^2}{2} \frac{1}{|\D_1|}\sum_{S\in\D_1\cup \D_2} \E_0 \left[ \sum_{j:A_j =S} \Gamma_j \right] \leq  \frac{\mu^2}{2} \frac{2s}{n} m \ .
\end{align*}

From this and \eqref{eqn:double_kl} we immediately get the result of the proposition.
\end{proof}

A closer look at the above proof gives an interesting insight. It appears the problem of interval estimation is as hard as the problem of detection.
Note that in essence the previous proof claims that estimating an interval is as hard as the problem of detection, that is, deciding between the hypotheses $H_0: \ S = \emptyset$ and $H_1: \ S \in \C$. In fact, the method proposed in Section~\ref{sec:int} already deals with this case, and exhibits the same performance if one ``adds'' the empty set to the class of $s$-intervals. The following proposition gives lower bounds both when considering $\P_S (\hat{S} \neq S)$ and $\E_S (|\hat{S} \triangle S|)$ as the error metric, that also captures the dependence on $\varepsilon$.

\begin{proposition} \label{prop:int_lower}
Let $\C$ be the class of $s$-intervals. Let $\varepsilon \in (0,1)$.
\begin{enumerate}
\item[(i)] if $\displaystyle{\max_{S \in \C\cup\emptyset}} \ \P_S (\hat{S} \neq S) \leq \varepsilon$ necessarily
\[
\mu \geq \sqrt{\frac{2n}{sm} \log \frac{1}{2 \varepsilon}} \ .
\]
\item[(ii)] if $\displaystyle{\max_{S \in \C\cup\emptyset}} \ \E_S (|\hat{S} \triangle S|) \leq \varepsilon$ necessarily
\[
\mu \geq \sqrt{\frac{2(n-s)}{sm} \left( \log\frac{n-s}{n+s} + \log\frac{s}{8\varepsilon} \right)}\ .
\]
\end{enumerate}
\end{proposition}
\begin{proof}
The assertion considering $\P_S (\hat{S} \neq S)$ as the error metric immediately follows from Theorem 3.1 in \cite{AS_Rui_2012}. This theorem is directly applicable as having an estimator $\hat{S}$ satisfying $(i)$ implies having a test $\Phi$ for the hypothesis testing problem
\[
H_0: \ S = \emptyset \qquad \textrm{versus} \qquad H_1: \ S \in \C
\]
with sum of type I and type II error probabilities no greater than $\varepsilon$.

As for the case when $\hat{S}$ satisfies $(ii)$, consider the following reduction of the problem. Let $\tilde{\C}$ denote the class that contains the empty set and the class of disjoint consecutive intervals defined in \eqref{eqn:disjoint_int}. For sake of simplicity assume that $n/s$ is an integer. It suffices to consider estimators of the form
\begin{equation}\label{eqn:estimator_union}
\hat{S} = \bigcup_{d\in D} S_d \ ,
\end{equation}
where $D$ is a subset of $\D$. In other words, the estimator is a (possibly empty) union of some disjoint $s$-intervals in $\D$. It is not restrictive to consider such estimators since if one has an arbitrary estimator $\hat{S}$ with expected number of errors at most $\varepsilon$, then we can define the estimator $\tilde{S}$ of the form in \eqref{eqn:estimator_union} which has error at most $4 \varepsilon$. For instance let $\tilde{S}$ be such that $S_d \subseteq \tilde{S}$ if and only if $|\hat{S} \cap S_d | \geq s/2$ for all $d\in\D$. Then $\E_S (|\tilde{S} \triangle S|) \leq 4 \varepsilon$ for all $S \in \C\cup \emptyset$.

Considering such estimators we can write the expected number of errors as
\[
\E_S (|\hat{S} \triangle S|) = \displaystyle{\sum_{i=1}^{n/s}} s \ \P_S (\textbf{1} \{ S_i \subseteq \hat{S} \} \neq \textbf{1} \{ S_i \subseteq S \} ) \ .
\]
This means that the above problem is similar to that of Proposition~\ref{prop:vanilla_lower} with a vector of length $n/s$, and support set size at most $1$ (but error bounded by $4\varepsilon/s$), concluding the proof.
\end{proof}

In the asymptotic regime for sparse signals we have the following corollary, which shows that the procedure proposed in Section~\ref{sec:int} is nearly optimal when considering both error metrics.

\begin{corollary}[$s$-intervals] \label{cor:int_lower}
Consider the setting of Proposition~\ref{prop:int_lower}, and suppose $s=o(n)$ as $n \to \infty$.
\begin{enumerate}
\item[(i)] if $\displaystyle{\lim_{n \to \infty} \ \max_{S \in \C\cup\emptyset}} \ \P_S (\hat{S} \neq S) = 0$ necessarily
\[
\mu \geq \omega_n \sqrt{\frac{n}{sm}} \ ,
\]
\item[(ii)] if $\displaystyle{\lim_{n \to \infty} \ \max_{S \in \C\cup\emptyset}} \ \P_S (\hat{S} \neq S) = 0$ necessarily
\[
\mu \geq \sqrt{\frac{2n}{sm} (\log s + \omega_n)} \ ,
\]
\end{enumerate}
where $\omega_n$ is an arbitrary sequence such that $\omega_n \to \infty$.
\end{corollary}


\subsubsection{Unions of $s$-intervals}

Consider again a slight modification of the class of interest, namely let $\C$ denote the class of all disjoint unions of $k$ or $(k-1)$ $s$-intervals. Similarly to the previous case, we reduce the problem to look like the general $s$-sparse case, and apply Proposition~\ref{prop:vanilla_lower}.

\begin{proposition} \label{prop:u_int_lower}
Let $\varepsilon \in (0,1)$ be fixed and suppose that $\displaystyle{\max_{S \in \C}} \ \E_S (|\hat{S} \triangle S|) \leq \varepsilon$. Then necessarily
\[
\mu \geq \sqrt{\frac{2(n-sk)}{sm} \left( \log k + \log\frac{n-sk}{n+s} + \log\frac{s}{8\varepsilon} \right)}\ .
\]
\end{proposition}
\begin{proof}
Assume again for sake of simplicity that $n/s$ is an integer and consider the class of consecutive $s$-intervals $\D$ defined in \eqref{eqn:disjoint_int}.
Let $\tilde{\D} \subset \C$ be the class which contains unions of $k$ or $(k-1)$ elements of $\D$. It suffices to consider only estimators that satisfy \eqref{eqn:estimator_union} since if there is a general estimator $\hat S$ for which $\max_{S\in\C} \E_S (|\hat S \triangle S|) \leq \varepsilon$ for all $S \in \C$ then there is an estimator $\tilde{S}$ of the form \eqref{eqn:estimator_union} satisfying $\max_{S\in\C} \E_S (|\tilde{S} \triangle S|) \leq 4\varepsilon$. Therefore the problem can once again be viewed as the unstructured case involving a vector of length $n/s$ and sparsity $k$ or $(k-1)$, and requiring that the estimator has expected Hamming-distance at most $4\varepsilon /s$. Using Proposition~\ref{prop:vanilla_lower} concludes the proof.
\end{proof}

\begin{corollary}[unions of $s$-intervals] \label{cor:u_int_lower}
Consider the setting of Proposition~\ref{prop:u_int_lower}, and suppose $sk = o(n)$ as $n\to \infty$. If there is an adaptive sensing and estimation strategy such that $\displaystyle{\lim_{n \to \infty} \ \max_{S \in \C}} \ \E_S (|\hat{S} \triangle S|) = 0$ then necessarily
\[
\mu \geq \sqrt{\frac{2n}{sm} (\log ks+\omega_n)} \ ,
\]
where $\omega_n$ is an arbitrary sequence for which $\omega_n \to \infty$.
\end{corollary}

The previous statements show the near-optimality of the procedure proposed in Section~\ref{sec:u_int}.


\subsubsection{$s$-stars and unions of $s$-stars}

The lower bounds for this class follow from similar arguments as the ones used for $s$-intervals by considering a maximal subclass of disjoint $s$-stars (meaning these do not share any edges). Let $N(p,s)$ be the size of such a subclass. We have the following lemma.
\begin{lemma}\label{lemma:disjoint_stars}
Let $N(p,s)$ denote the maximal number of disjoint stars of size $s$ in a complete graph with $p$ vertices. We have
\[
N(p,s) \geq \frac{p(p-1-s)}{2s} \ .
\]
\end{lemma}
With this in mind we can get a performance lower bound, proved in an analogous way to that of Proposition~\ref{prop:int_lower_v2}.
\begin{proposition} \label{prop_star_lower_v2}
Let $\C$ be the class of $s$-stars. Assume $\displaystyle{\max_{S \in \C}} \ \P_S (\hat{S} \neq S) \leq \varepsilon$. Then necessarily
\[
\mu \geq (1- \varepsilon) \sqrt{\frac{N(p,s)}{2m}} \ .
\]
\end{proposition}
We can also get results analogous to Propositions~\ref{prop:int_lower} and \ref{prop:u_int_lower}, and Corollaries~\ref{cor:int_lower} and \ref{cor:u_int_lower}. We only state the results for the unions of $s$-stars, which shows the near-optimality of the proposed procedure.
\begin{proposition} \label{prop:u_star_lower}
Let $\C'$ denote the class of unions of $k$ or $k-1$ disjoint $s$-stars. Let $\varepsilon \in (0,1)$ be fixed and suppose that $\displaystyle{\max_{S \in \C}} \ \E_S (|\hat{S} \triangle S|) \leq \varepsilon$. Then necessarily
\[
\mu \geq \sqrt{\frac{2(N(p,s) -k)}{m} \left( \log k + \log \frac{N(p,s)-k}{N(p,s)+1} + \log \frac{s}{8\varepsilon}\right)} \ .
\]
\end{proposition}

\begin{corollary}[unions of $s$-stars] \label{cor:u_star_lower}
Consider the setting of Proposition~\ref{prop:u_star_lower}, and suppose $ks = o(n)$ as $n\to\infty$. If there is an adaptive sensing and estimation strategy such that $\displaystyle{\lim_{n \to \infty} \ \max_{S \in \C}} \ \E_S (|\hat{S} \triangle S|) = 0$ then necessarily
\[
\mu \geq \sqrt{\frac{2n}{sm} (\log ks+\omega_n)} \ ,
\]
where $\omega_n$ is an arbitrary sequence for which $\omega_n \to \infty$.
\end{corollary}


\subsubsection{$s$-submatrices}

Again, akin to the $s$-intervals and $s$-stars we can get lower bounds without including the empty set to the class when considering the probability of error as the metric of interest.
\begin{proposition} \label{prop:submat_lower_v2}
Let $\C$ be the class of submatrices, and suppose $n_1 ,n_2 ,s$ are such that we can cover the matrix $M$ entirely with disjoint submatrices of size $s$. Let $\varepsilon>0$ and suppose $\displaystyle{\max_{S \in \C}} \ \P_S (\hat{S} \neq S) \leq \varepsilon$. Necessarily
\[
\mu \geq (1- \varepsilon) \sqrt{\frac{n}{2sm}} \ .
\]
\end{proposition}

We can once more derive results when including the empty set in the class.
\begin{proposition} \label{prop:submat_lower}
Suppose $n_1 ,n_2 ,s$ are such that we can cover the matrix $M$ entirely with disjoint submatrices of size $s$, and fix $\varepsilon \in (0,1)$.
\begin{enumerate}
\item[(i)] if $\displaystyle{\max_{S \in \C \cup \emptyset}} \ \P_S (\hat{S} \neq S) \leq \varepsilon$, then necessarily
\[
\mu \geq \sqrt{\frac{2n}{sm} \log \frac{1}{2 \varepsilon}} \ ,
\]
\item[(ii)] if $\displaystyle{\max_{S \in \C \cup \emptyset}} \ \E_S (|\hat{S} \triangle S|) \leq \varepsilon$, then necessarily
\[
\mu \geq \sqrt{\frac{2(n-s)}{sm} \left( \log\frac{n-s}{n+s} + \log\frac{s}{8\varepsilon} \right)} \ .
\]
\end{enumerate}
\end{proposition}

The condition about the relation between $n_1 ,n_2$ and $s$ in the previous propositions is merely to simplify presentation. This can be easily relaxed by bounding the number of disjoint submatrices of size $s$ in a matrix of size $n_1 \times n_2$. For instance, one such bound is $\max \{ n_1 \cdot \lfloor n_2 /s \rfloor , n_2 \cdot \lfloor n_1 /s \rfloor \}$. This condition does not play a role when considering the behavior of the bound in the asymptotic regime for sparse signals in the following corollary.
\vspace{1pt}
\begin{corollary}[$s$-submatrices] \label{cor:submat_lower}
Consider the setting of Proposition~\ref{prop:submat_lower}, and suppose $s=o(n)$ as $n\to\infty$.
\begin{enumerate}
\item[(i)] if $\displaystyle{\lim_{n \to \infty} \ \max_{S \in \C \cup \emptyset}} \ \P_S (\hat{S} \neq S) = 0$, then
\[
\mu \geq \omega_n \sqrt{\frac{2n}{sm}} \ ,
\]
\item[(ii)] if $\displaystyle{\lim_{n \to \infty} \ \max_{S \in \C \cup \emptyset}} \ \E_S (|\hat{S} \triangle S|) = 0$, then
\[
\mu \geq \sqrt{\frac{2n}{sm} \left( \log s + \omega_n \right)} \ ,
\]
\end{enumerate}
where $\omega_n$ is an arbitrary sequence such that $\omega_n \to \infty$.
\end{corollary}

The previous results show the near optimality of the procedure proposed in Section~\ref{sec:submat}, in the case when $n_1 ,n_2$ are the same order of magnitude. When either of them is close to $n$ the problem becomes similar to the unstructured case (however, this is not captured by the corollary above).



\section{Final Remarks} \label{sec:conclusion}

In this work we investigated the problem of support estimation of structured sparse signals under the adaptive sensing paradigm. These results broaden our understanding of the fundamental limits of adaptive sensing and also provide a practical method for estimating signal supports. The procedure suggested in this work is rather general and simple, and also turns out to be near-optimal in a variety of interesting cases. It is important to point out that the proposed procedure requires knowledge of some parameters of the problem that might not be available in a real-life setting. Also, neither the fundamental performance limits nor the performance of the proposed procedure are yet fully understood for arbitrary classes of signal support sets. These might prove to be interesting areas for future research.

\appendix

\section*{Appendix}
\begin{proof}[Proof of Proposition~\ref{prop:SLRT}]

To ease notation we write $N\equiv N_{\Gamma}$. The proof of all the statements in the proposition hinges on the derivation of upper bounds for the expected value of the stopping time $N$. Recall the definition of the log-likelihood ratio
$$\bar{z}_k = \displaystyle{\sum_{i=1}^{k}} \log \frac{f_1 (y_i)}{f_0 (y_i)} = \displaystyle{\sum_{i=1}^{k}} z_i \ ,$$
where $z_i=\frac{\Gamma}{2}\mu(2y_i-\mu).$ From Wald's identity \cite{SLRT_Wald_1945} we know that
$$\E (\bar z_N) = \E (N) \E (z_1) \ .$$
Since it is easy to compute $\E (z_1)$ directly, in order to control $\E(N)$ we need to control $\E (\bar z_N)$. Note that ${\E_0 (z_1) < 0 < \E_1 (z_1)}$, thus to get an upper bound on $\E_0(N)$ we need to lower bound $\E_0(\bar{z}_N)$, and to get an upper bound on $\E_1(N)$ we need to upper bound $\E_1(\bar{z}_N)$. In what follows assume $H_0$ is true, as the case for $H_1$ is entirely analogous. Our proof hinges on the following technical lemma.
\begin{lemma}\label{lemma:tech}
\begin{equation}\label{eqn:tech1}
l+\E_0 (z_1 | z_1 \leq 0) \ \leq\  \E_0 (\bar{z}_N | \bar{z}_N \leq l) \ \leq\  l\ ;
\end{equation}
\begin{equation}\label{eqn:tech2}
e^l \E_0\left(\left.e^{z_1}\right|z_1\leq 0\right) \ \leq\  \E_0 \left(\left.e^{\bar{z}_N} \right| \bar{z}_N \leq l\right) \ \leq\  e^l\ ;
\end{equation}
\begin{equation}\label{eqn:tech3}
e^u \ \leq\  \E_0 \left(\left.e^{\bar{z}_N} \right| \bar{z}_N \geq u\right) \ \leq\  e^u \E_0\left(\left.e^{z_1}\right|z_1\geq 0\right)\ .
\end{equation}
\end{lemma}
\begin{proof}
We prove only the first statement, as the proof of the other two statements follow with essentially the same reasoning. First note that for any normal random variable $\xi \sim \mathcal{N} (\nu , \sigma^{2})$ and $c \leq 0$ we have
\begin{equation}\label{eqn:shift}
\E (\xi -c | \xi \leq c) \geq \E (\xi | \xi \leq 0)\ .
\end{equation}
This can be justified by writing the conditional densities of $\xi - c | \xi \leq c$ and $\xi | \xi \leq 0$, respectively
\begin{align*}
f_{\xi -c | \xi \leq c} (x) &= K_1 e^{- \frac{( (x- \nu) +c )^2}{2 \sigma^{2}}} \1 \{ x \leq 0 \} \\
f_{\xi | \xi \leq 0} (x) &= K_2 e^{- \frac{(x- \nu)^2}{2 \sigma^{2}}} \1 \{ x \leq 0 \}\ ,
\end{align*}
where $K_2>K_1>0$ are the appropriate normalization constants. It is easy to show that these densities satisfy
\begin{align*}
f_{\xi -c | \xi \leq c}(x) \leq f_{\xi | \xi \leq 0} (x) &\text{ if } x\leq x_0\\
f_{\xi -c | \xi \leq c}(x) \geq f_{\xi | \xi \leq 0} (x) &\text{ if } x\geq x_0\ ,
\end{align*}
where $x_0$ is simply given by
$$x_0 = \frac{2\sigma^2 \log \frac{K_1}{K_2} -c^2}{2c}+\nu\ .$$
This, in turn implies \eqref{eqn:shift}, as
\begin{align*}
\lefteqn{\E (\xi -c | \xi \leq c) - \E (\xi | \xi \leq 0) = \int x f_{\xi -c | \xi \leq c}(x) \d x - \int x f_{\xi | \xi \leq 0}(x) \d x}\\
& = \int x \left(f_{\xi -c | \xi \leq c}(x)-f_{\xi | \xi \leq 0}(x)\right) \d x - x_0\int f_{\xi -c | \xi \leq c}(x)-f_{\xi | \xi \leq 0}(x) \d x\\
&= \int (x-x_0) \left(f_{\xi -c | \xi \leq c}(x)-f_{\xi | \xi \leq 0}(x)\right) \d x\\
&\geq 0 \ .
\end{align*}

We are now ready to prove the lemma. First note that
$$z_1 = \log \frac{f_1 (y_1)}{f_0 (y_1)} = \Gamma \mu y_1 - \frac{\Gamma \mu^{2}}{2} \stackrel{\smaller H_0}{\sim} \normal\left( -\frac{\Gamma}{2} \mu^{2}, \Gamma \mu^{2} \right)\ .$$
Therefore
\begin{align*}
\E_0\Big(\bar z_N\Big|\bar z_N \leq l\Big) &= \E_0\Big(\E_0\left(\bar z_N\left| N, \bar z_{N-1} , \bar z_N \leq l\right.\right)\Big|\bar z_N \leq l\Big)\\
&= l+\E_0\Big(\E_0\left(z_N -(l-\bar z_{N-1})\left| N, \bar z_{N-1} , z_N\leq l-\bar z_{N-1}\right.\right)\Big|\bar z_N \leq l\Big)\\
&\geq l+\E_0\Big(\E_0\left(z_N \left| N, \bar z_{N-1} , z_N\leq 0\right.\right)\Big|\bar z_N \leq l\Big)\\
&=  l+\E_0\Big(\E_0\left(z_1\left| N, \bar z_{N-1} , z_1\leq 0\right.\right)\Big|\bar z_N \leq l\Big)\\
&=  l+\E_0\Big(z_1\left|z_1\leq 0\right.\Big)\ ,
\end{align*}
where the inequality follows from \eqref{eqn:shift}, concluding the proof of statement \eqref{eqn:tech1}. The other two statements are shown in a similar fashion, by noting also that the exponential function is monotone increasing.
\end{proof}

With the lemma result at hand, note that
\begin{align*}
\E_0 (\bar{z}_N) &= \alpha_{\Gamma} \E_0 (\bar{z}_N | \bar{z}_N \geq u) + (1- \alpha_{\Gamma}) \E_0 (\bar{z}_N | \bar{z}_N \leq l)\\
&\geq  \alpha_{\Gamma} u + (1- \alpha_{\Gamma}) l + (1- \alpha_{\Gamma}) \E_0 (z_1 | z_1 \leq 0) \ ,
\end{align*}
where the last step follows simply from the lemma. Using this together with Wald's inequality yields
\begin{align*}
\Gamma \E_0 (N) & \leq  - \frac{2}{\mu^{2}} \Big( \alpha_{\Gamma} u + (1- \alpha_{\Gamma}) l + (1- \alpha_{\Gamma}) \E_0 (z_1 | z_1 \leq 0) \Big) \\
&= \frac{2}{\mu^{2}} \Big( \alpha_{\Gamma} \log \frac{\alpha}{1-\beta} + (1- \alpha_{\Gamma}) \log \frac{1-\alpha}{\beta} - (1- \alpha_{\Gamma}) \E_0 (z_1 | z_1 \leq 0) \Big) \ .
\end{align*}
So, provided we can show that $\alpha_{\Gamma} \to \alpha$ and $\E_0 (z_1 | z_1 \leq 0) \to 0$ as $\Gamma \to 0$ the statement of the proposition follows, as we obtain the same limit as in \eqref{eqn:prec_lower_0}.

Note first that $\P_0(z_1\leq 0)=\Phi(\mu\sqrt{\Gamma}/2)\to 1/2$ as $\Gamma\to 0$, where $\Phi$ denotes the standard normal cumulative distribution function. Now, since $-|z_1|\leq z_1\1\{z_1\leq 0\}\leq 0$ we conclude that $\E_0(z_1\1\{z_1\leq 0\})\to 0$ when $\Gamma\to 0$, since $\E_0(|z_1|)\leq\sqrt{\E_0(z_1^2)}\to 0$. Therefore $\E_0(z_1|z_1\leq 0)=\E_0(z_1\1\{z_1\leq 0\})/\P_0(z_1\leq 0)\to 0$.

To conclude the proof we need to show that $\alpha_{\Gamma} \to \alpha$ as $\Gamma \to 0$. We can check this using the moment generating function of $\bar{z}_N$. Begin by noting that
\begin{align*}
1&= \E_0 \left(\prod_{i=1}^{N} \frac{f_1 (y_i)}{f_0 (y_i)} \right) = \E_0 (e^{\bar{z}_N}) = (1- \alpha_{\Gamma} ) \E_0 \left(\left.e^{\bar{z}_N} \right| \bar{z}_N \leq l\right) + \alpha_{\Gamma} \E_0 \left(\left.e^{\bar{z}_N} \right| \bar{z}_N \geq u\right)  \ .
\end{align*}
Hence
\begin{equation}\label{eqn:alpha_converges}
\alpha_{\Gamma} = \frac{1 - \E_0 \left(\left.e^{\bar{z}_N} \right| \bar{z}_N \leq l\right)}{\E_0 \left(\left.e^{\bar{z}_N} \right| \bar{z}_N \geq u\right) - \E_0 \left(\left.e^{\bar{z}_N} \right| \bar{z}_N \leq l\right)} \ .
\end{equation}
We can now use the statements \eqref{eqn:tech2} and \eqref{eqn:tech3} of Lemma~\ref{lemma:tech}.

It can be easily shown that $\E_0(e^{z_1}|z_1\leq 0) \to 1$ and $\E_0(e^{z_1}|z_1\geq 0) \to 1$ as $\Gamma \to 0$. Therefore from Lemma~\ref{lemma:tech} we get that
$$\E_0 \left(\left.e^{\bar{z}_N} \right| \bar{z}_N \leq l\right) \to e^l\ ,$$
and
$$\E_0 \left(\left.e^{\bar{z}_N} \right| \bar{z}_N \geq u\right) \to e^u$$
as $\Gamma\to 0$. This, together with \eqref{eqn:alpha_converges} concludes the proof of the first statement of the proposition. The proof of the second statement in entirely analogous.
\end{proof}

\begin{proof}[Proof of Proposition~\ref{prop:prec_symm_classes}]
The maximum of the quantity above is attained when $||b||_1 = m$, so we will assume this in what follows. For a fixed $i \in \{ 1, \dots ,n \}$ let $\C_i = \{ S \in \C : \ i \in S \}$. In case of symmetric classes we have that $|\C_i|=c$ does not depend on $i$. Also note that $c/ |\C| = s/n$. To see the latter consider a random coordinate $J$ which is uniform on $\{ 1, \dots ,n \}$, and a random coordinate $K$, which is selected sequentially: first select $S \in \C$ uniformly at random, then select $i \in S$ uniformly at random. When $\C$ is symmetric the distribution of $J$ and $K$ are the same, and
\[
\frac{1}{n} = \P (J=i) = \P (K=i) = \frac{c}{|\C|} \frac{1}{s}
\]
With this we can write
\begin{align*}
\displaystyle{\min_{S \in \C} \ \sum_{S' \in \C \setminus \{ S \}} \ \sum_{i \in S \triangle S'}} \ b_i & = \displaystyle{\min_{S \in \C} \ \sum_{S' \in \C \setminus \{ S \}}} \left( \displaystyle{\sum_{i \in S \setminus S'}} \ b_i + \displaystyle{\sum_{i \in S' \setminus S}} \ b_i \right) \\
& = \displaystyle{\min_{S \in \C}} \left( \displaystyle{\sum_{i \in S} \ \sum_{S' \in \C \setminus \{ S \}}} \1 \{ i \notin S' \} \ b_i + \displaystyle{\sum_{i \notin S} \ \sum_{S' \in \C \setminus \{ S \}}} \1 \{ i \in S' \} \ b_i \right) \\
& = \displaystyle{\min_{S \in \C}} \left( \displaystyle{\sum_{i \in S}} (|\C| -c) \ b_i + \displaystyle{\sum_{i \notin S}} c \ b_i \right) \\
& = \displaystyle{\min_{S \in \C}} \left( (|\C| -c) \ b_S + c \ (m- b_S ) \right) \\
& = cm + (|\C| - 2c) \ \displaystyle{\min_{S \in \C}} \ b_S \ ,
\end{align*}
where $b_S = \sum_{i \in S} b_i$. However,
\begin{align*}
\displaystyle{\min_{S \in \C}} \ b_S &\leq \frac{1}{|\C|} \displaystyle{\sum_{S \in \C}} \ b_S = \frac{1}{|\C|} \displaystyle{\sum_{S \in \C} \ \sum_{i \in S}} \ b_i \\
& = \frac{1}{|\C|} \displaystyle{\sum_{i=1}^{n}} \1 \{ i \in S \} b_i = \frac{1}{|\C|} c \displaystyle{\sum_{i=1}^{n}} b_i \\
& = \frac{cm}{|\C|} \ .
\end{align*}
Now note that when $b_i = m/n$ for all $i=1,\dots ,n$, we have $b_S = s m /n = cm / |\C|$ for all $S \in \C$.
\end{proof}

\begin{proof}[Proof of Lemma~\ref{lemma:disjoint_stars}]
Consider the following sequence of graphs denoted by $G_0 ,G_1 ,\ldots, G_K$, where $K\in\N$. Let $G_0$ denote the graph with $p$ vertices and no edges. The graphs $G_1,\ldots,G_K$ are obtained recursively by adding a disjoint star of  $s$ to the graph until this is no longer possible. In other words, for $k\in\{1,\ldots,K\}$ the graph $G_k$ is constructed by adding a disjoint $s$-star of $G_{k-1}$. Let $d_k (v)$ denote the degree of $v \in G_k$. Notice that for any $k \in \{0,\ldots,K\}$ if there exists $v \in G_k$ such that $d_k (v) < p-1-s$ we can add a star to $G_k$ centered in $v$. This means that for the index $K$ we have that $d_K (v) \geq p-1-s$ for all $v \in G_K$. Thus the graph $G_K$ has at least $p(p-1-s)/2$ edges and is built entirely of disjoint stars of size $s$. The statement now follows.
\end{proof}

\bibliographystyle{plain}
\bibliography{AS_references}

\end{document}